\documentclass{imsart}
\arxiv{2406.18397}
\usepackage[
  margin=2cm,
  includefoot,
  footskip=30pt,
]{geometry}
\usepackage{amsmath,amsthm,amssymb}
\usepackage[utf8]{inputenc}
\usepackage[T1]{fontenc}
\usepackage[english]{babel}
\usepackage{dsfont}
\usepackage[table]{xcolor}
\usepackage{hyperref}
\definecolor{burgundy}{rgb}{0.5, 0.0, 0.13}
\definecolor{chamoisee}{rgb}{0.63, 0.47, 0.35}
\hypersetup{colorlinks=true,linkcolor=burgundy,citecolor=chamoisee,urlcolor=burgundy}
\usepackage{natbib}
\usepackage{tikz}
\usepackage{booktabs}

\newtheorem{theorem}{Theorem}%
\newtheorem{proposition}{Proposition}
\newtheorem{lemma}[proposition]{Lemma}

\newtheorem{remark}{Remark}

\def\beq{\begin{equation}}
\def\eq{\begin{equation}}
\def\eeq{\end{equation}}
\def\qe{\end{equation}}
\def\beqn{\begin{eqnarray*}}
\def\eeqn{\end{eqnarray*}}
\def\bitem{\begin{itemize}}
\def\eitem{\end{itemize}}
\def\benum{\begin{enumerate}}
\def\eenum{\end{enumerate}}
\def\bmult{\begin{multline*}}
\def\emult{\end{multline*}}
\def\bcenter{\begin{center}}
\def\ecenter{\end{center}}

\DeclareMathOperator*{\argmax}{arg\, max}
\DeclareMathOperator*{\argmin}{arg\, min}

\DeclareMathOperator{\I}{Id}
\DeclareMathOperator{\diag}{diag}

\def\bU{\boldsymbol{U}}

\def\bb{\mathbf{b}}

\def\bbE{\mathds{E}}

\def\bbH{\mathds{H}}

\def\bbR{\mathds{R}}

\def\bb\bU{\mathds{\bU}}

\newcommand{\E}{\operatorname{\mathds{E}}}
\newcommand{\leb}{ {\rm Leb}}
\renewcommand{\P}{\operatorname{\mathds{P}}}
\renewcommand{\phi}{\varphi}
\renewcommand{\bar}{\overline}
\renewcommand{\hat}{\widehat}
\renewcommand{\tilde}{\widetilde}
\newcommand{\Var}{\operatorname{Var}}
\newcommand{\Cov}{\operatorname{Cov}}

\def\\bUnif{\text{\bUnif}}

\newcommand\ind{{\mathbf 1}}
\newcommand{\1}{\mathds{1}}

\def\Perm{{\mathfrak S}}

\newcommand\R{{\mathds R}}

\begin{document}

\begin{frontmatter}

\title{Second Maximum of a Gaussian~Random~Field\\ and Exact ($t$-)Spacing test}

\runtitle{Generalized $t$-Spacing test}

\begin{aug}
\author[A]{\fnms{Jean-Marc} \snm{Aza\"is}\ead[label=e1]{jean-marc.azais@math.univ-toulouse.fr}} 
\author[B]{\fnms{Federico} \snm{Dalmao}\ead[label=e2]{fdalmao@unorte.edu.uy}}
\author[C,D]{\fnms{Yohann} \snm{De Castro}\ead[label=e3]{yohann.de-castro@ec-lyon.fr}
}

\address[A]{Institut de Math\'ematiques de Toulouse\\ Universit\'e de Toulouse, France\printead[presep={,\ }]{e1}}
\address[B]{DMEL, CENUR Litoral Norte\\  Universidad de la República, Salto, Uruguay\printead[presep={,\ }]{e2}}
\address[C]{Institut Camille Jordan, CNRS UMR 5208\\  École Centrale Lyon, France\printead[presep={,\ }]{e3}}
\address[D]{Institut Universitaire de France (IUF)}

\runauthor{Aza\"is, Dalmao, and De Castro}
\end{aug}

\begin{abstract}
{\ }In this article, we introduce the novel concept of the second maximum of a Gaussian random field defined on a Riemannian manifold: Once the field reaches its global maximum at a point~$t_1$, the second maximum is the largest value of the field away from~$t_1$, after the value and the gradient of the field at~$t_1$ have been removed. This statistic serves as a powerful tool for characterizing the distribution of the global maximum. By employing an ad-hoc Kac-Rice formula, we derive the explicit distribution of the maximum, conditioned on the second maximum and the independent part of the Riemannian Hessian (the part of the Hessian at~$t_1$ that is independent of the value and the gradient of the field at~$t_1$). This approach yields an exact test based on the spacing between these maxima, which we refer to as the spacing test. 

We investigate the applicability of this test for detecting sparse alternatives in Gaussian symmetric tensors, continuous sparse deconvolution, and two-layered neural networks with smooth rectifiers. Theoretical results are supported by numerical experiments that illustrate the calibration and power of the proposed procedures. More generally, we provide a framework for applying the spacing test to continuous sparse regression for any Gaussian random field on a Riemannian manifold. 

Furthermore, when the variance-covariance function is known only up to a scaling factor, we derive an exact Studentized version of the procedure, coined the t-spacing test. We show that this test is perfectly calibrated under the null hypothesis and exhibits high power for detecting sparse alternatives. As a corollary, we derive an exact, non-asymptotic test for the spiked tensor model without requiring prior knowledge of the noise level.
\end{abstract}



\end{frontmatter}

\maketitle

{
\it
\footnotesize
\begin{center}
Version of \today
\end{center}
}

\section{Introduction}

\subsection{The $(t$-$)$spacing test and the second maximum}
This paper introduces a test for the mean $m(\cdot)$ of a Gaussian random field~$Z(\cdot)$, defined on a~$\mathcal C^2$-compact Riemannian manifold ${ M}$ of dimension $d$ without boundary, which can be decomposed~as
\begin{equation}
\label{e:Z_process}
\forall t\in  M\,,\quad
Z(t) = m(t) + \sigma X(t)\,,
\end{equation}
where $m(\cdot)$ is any $\mathcal C^2$-function, $\sigma>0$ is the standard deviation, and~$X(\cdot)$ is a centered Gaussian random field such that $\Var[X(t)]=1$ for every $t\in { M}$. We consider the simple global null hypothesis~$\mathds H_0\,:\,\text{‘‘}m(\cdot)=0\,\text{''}$, even when the standard error~$\sigma$ is unknown.  The test investigated is the quantile~$\hat\alpha$ of the maximum $\lambda_1$ under the null defined as
\begin{equation}
\label{eq:lambda1_t1_intro}
\lambda_1=\max_{t\in  M}\{\sigma X(t)\}
\end{equation}
and we denote by $t_1\in M$ its argument maximum.

In the applications we consider, testing this global null hypothesis is the natural first question: one asks whether there is a signal at all (a spike in a noisy tensor, a source in a mixture, a hidden unit in a network) before asking where it is. When the test rejects, the pair $(t_1,\lambda_1)$ gives a candidate location and intensity for the signal, see Section~\ref{sec:examples}. The complementary question, testing the mean at given locations with a correction for multiple comparisons, is not treated here; in the discrete case, it is addressed by the multiple spacing tests of \citet{azais2022multiple}.

\medskip

\newcommand\sizespec{4.6cm}

 \begin{figure}[!t]
\begin{centering}
\includegraphics[height=\sizespec]{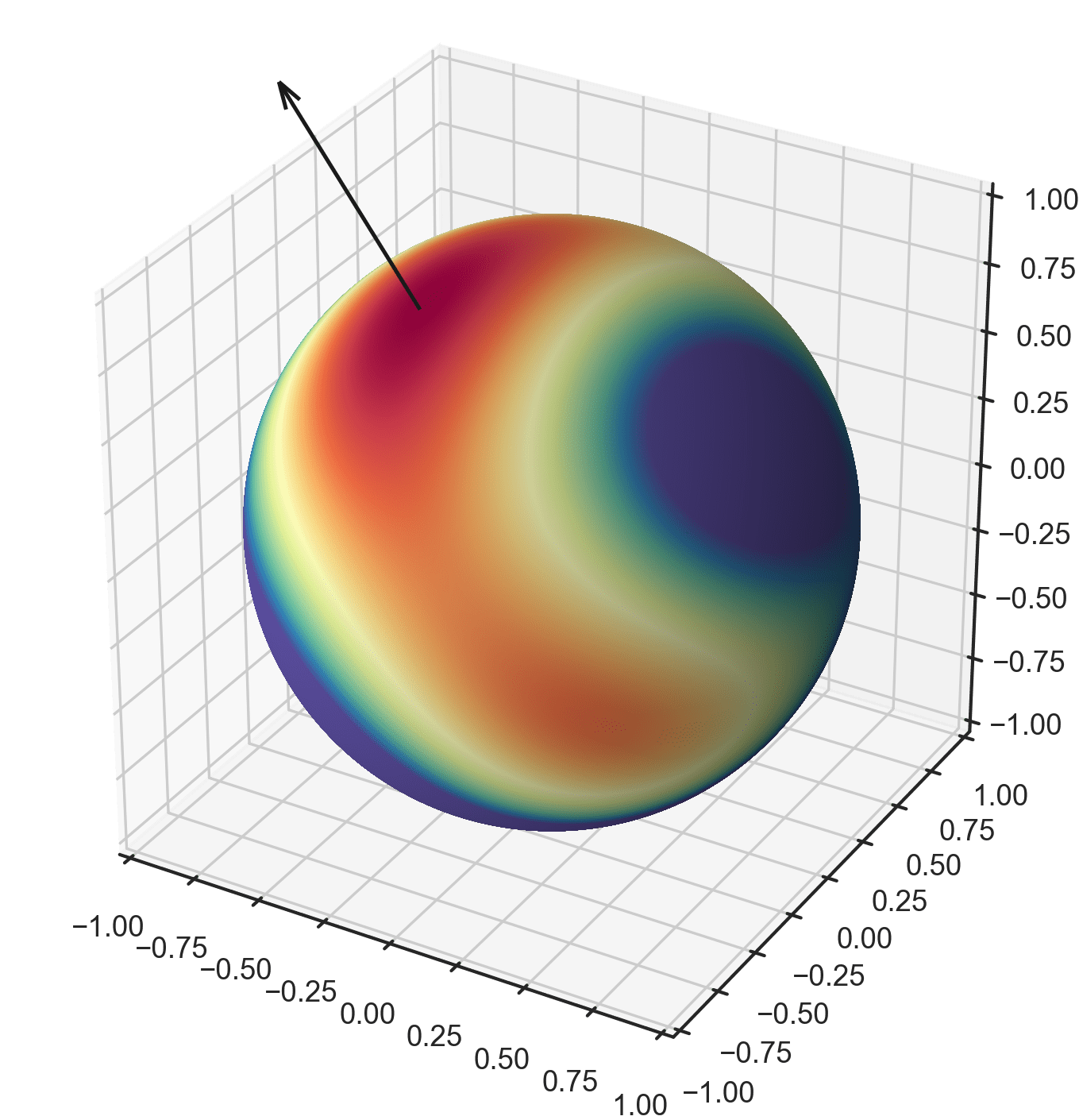}
\hspace*{0.2cm}
\includegraphics[height=\sizespec]{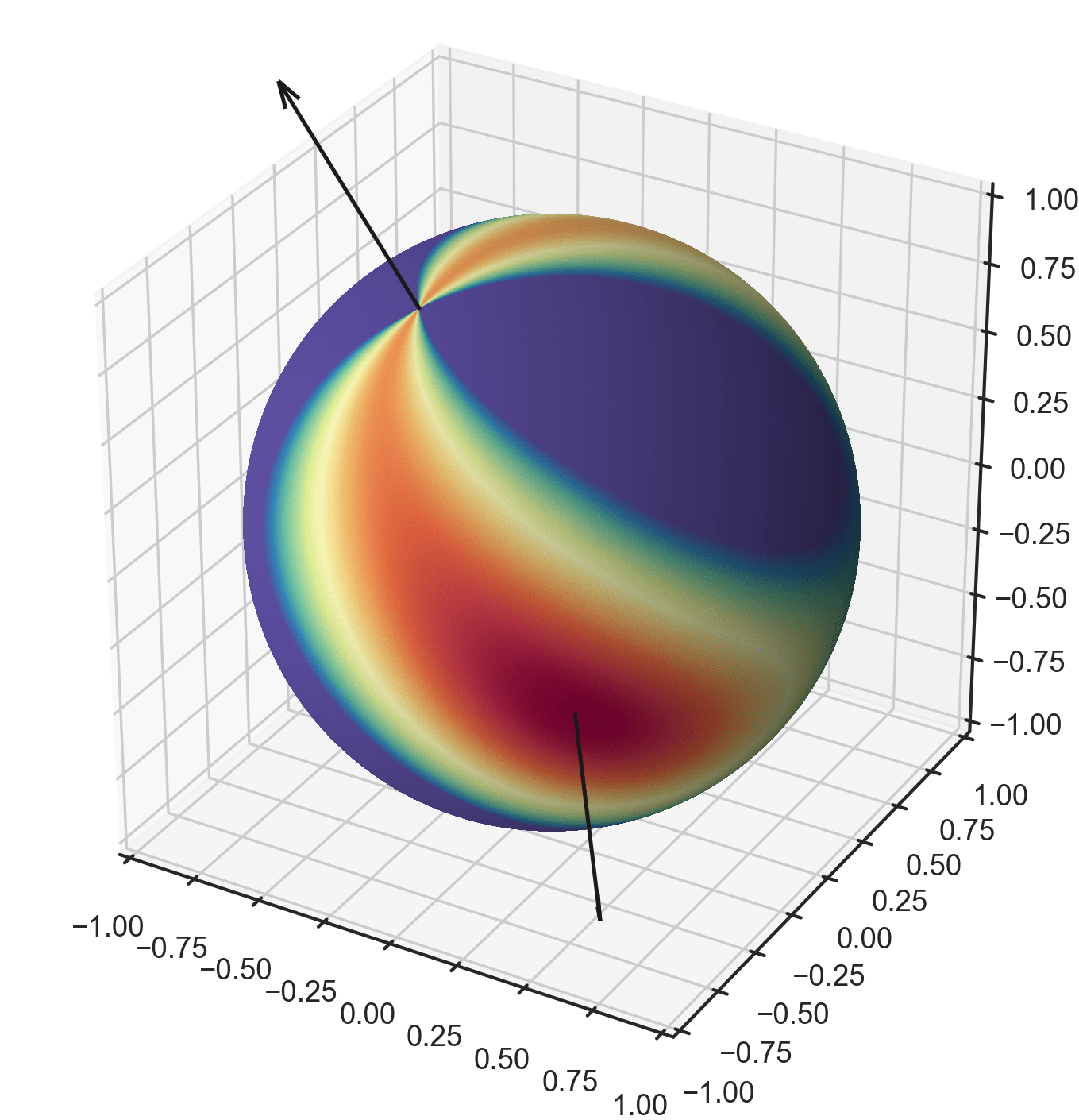}
\hspace*{0.2cm}
\includegraphics[height=\sizespec]{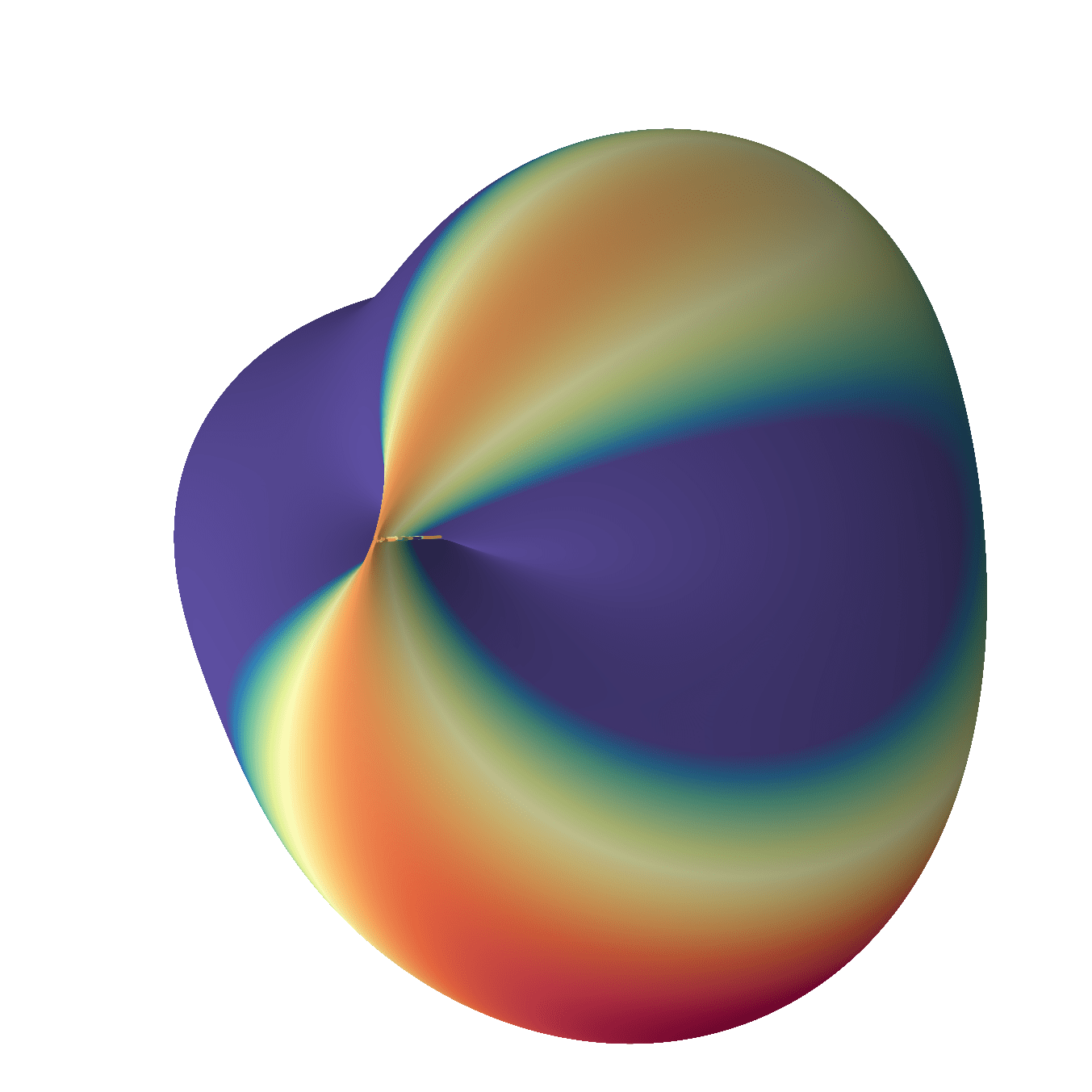}
\end{centering}
\caption{{\bf [Spiked tensor PCA, Section~\ref{sec:spike_tensor_model}, example 1/5]} 
Visualisation of the random fields on the sphere $\mathds S^2$ for the Spiked Tensor PCA problem. \textbf{Left:} The Gaussian homogeneous polynomial $X(\cdot)$, where the arrow indicates the global maximizer $t_1$ (first eigenvector). \textbf{Middle:} The conditional random field $X^{|t_1}(\cdot)$ defined in \eqref{e:Xbarra_intro}, where the arrow indicates the second maximizer $t_2$. \textbf{Right:} A volumetric view of $X^{|t_1}(\cdot)$ where the radial height represents the field value.}
\label{fig:eigenvector}
\end{figure}

This quantile is built from the distribution of~$\lambda_1$ conditional on the values of the so-called second maximum $\lambda_2$ (the largest value of the field once its value and gradient at the maximiser~$t_1$ are removed, defined formally in~\eqref{e:Xbarra_intro} below) and the so-called independent part of the Riemannian Hessian~$\Omega$ (a symmetric random matrix built from the Hessian of the field at~$t_1$ that is independent of $(X(t_1),\nabla X(t_1))$; see~\eqref{def:R}), under the null. Suppose, in a first step, that $\sigma $ is known, in such a case we will demonstrate that the cumulative distribution of this conditional law can be expressed as a ratio, resulting in the following expression derived from an ad-hoc {K}ac-{R}ice formula,
\begin{equation}
\label{e:hat_alpha}
\hat\alpha 
:=\frac{\displaystyle\int_{{\lambda_1}/\sigma}^{+\infty}\!\!\!\!\!\!\det ( u \mathrm{Id}-\Omega/\sigma)    \phi(u )\, \mathrm{d}u}
{\displaystyle\int_{{\lambda_2}/\sigma}^{+\infty}\!\!\!\!\!\!\det ( u \mathrm{Id}-\Omega/\sigma)    \phi(u )\, \mathrm{d}u}
\,,
\end{equation}
where $\varphi(\cdot)$ is the standard Gaussian density.

\medskip

We will show the exactness of this test, meaning that $\hat\alpha$ is uniformly distributed on the interval~$(0,1)$ under the null hypothesis. It demonstrates that one minus the ratio \eqref{e:hat_alpha} represents the cumulative distribution function (CDF) of the law of $\lambda_1$ conditional on $\lambda_2$ and $\Omega$, under the null hypothesis. Small values of~$\hat\alpha$ indicate that the maximum $\lambda_1$ is abnormally large compared to the value of the second maximum $\lambda_2$, making $\hat\alpha$ interpretable as a $p$-value. This test is referred to as the ‘‘spacing test''. The second maximum is defined as 
\begin{align} 
	\label{e:Xbarra_intro}
	\lambda_2
		&:=\max_{s\neq t_1\in  M}\, \{\sigma X^{|t_1}(s)\}\,,\\ \notag
	\text{ with, for }s\neq t\,,\  X^{|t}(s) 
		&:= \frac{X(s)-c(s,t)\,X(t)-\nabla_t c(s,t)^\top\Lambda_2^{-1}(t)\nabla X(t)}{1-c(s,t)}
\end{align}
where $c(\cdot,\cdot)$ is the variance-covariance function of $X(\cdot)$ and $\Lambda_2(t) := \Var[\nabla X(t)]$, the formal definition will be given in Section~\ref{sec:regression_remainders}. Note that $X^{|t}(s)$ is a normalisation of the remainder of the regression of~$X(s)$ with respect to $(X(t),\nabla X(t))$. Intuitively, $X^{|t}(s)$ removes from $X(s)$ the part explained by the value and the gradient of the field at~$t$, and renormalises the remainder; the second maximum $\lambda_2$ is then the largest value of this residual field. This is the continuous counterpart of the second knot of the LARS path, once the first knot~$t_1$ has entered (see the middle panel of Figure~\ref{fig:eigenvector}). An example of the second maximum in the rank-one tensor detection case, referred to as Spiked tensor PCA, is illustrated in Figure~\ref{fig:eigenvector} and developed in Section~\ref{sec:spike_tensor_model}. 

\medskip

Furthermore, we will show how to achieve the same result when the variance~$\sigma^2$ is \emph{unknown} using an estimation $\hat\sigma^2$ built from the Karhunen-Loève expansion of $X^{|t_1}(\cdot)$, see the right panel of Figure~\ref{fig:eigenvector} for an illustration of this random field in the case of rank-one tensor detection. We refer to this test as the $t$-spacing test which detects abnormally large spacing between~$\lambda_1/\hat \sigma$ and~$\lambda_2/\hat\sigma$. Note that $\lambda_1$ and $\lambda_2$ are the same for the spacing and the $t$-spacing test, and these values can be computed without knowing~$\sigma$. These results are supported by numerical experiments as illustrated in Figure~\ref{fig:cdf_spacings}. 

\renewcommand\sizespec{5.7cm}

\begin{figure}[!ht]
\includegraphics[height=\sizespec]{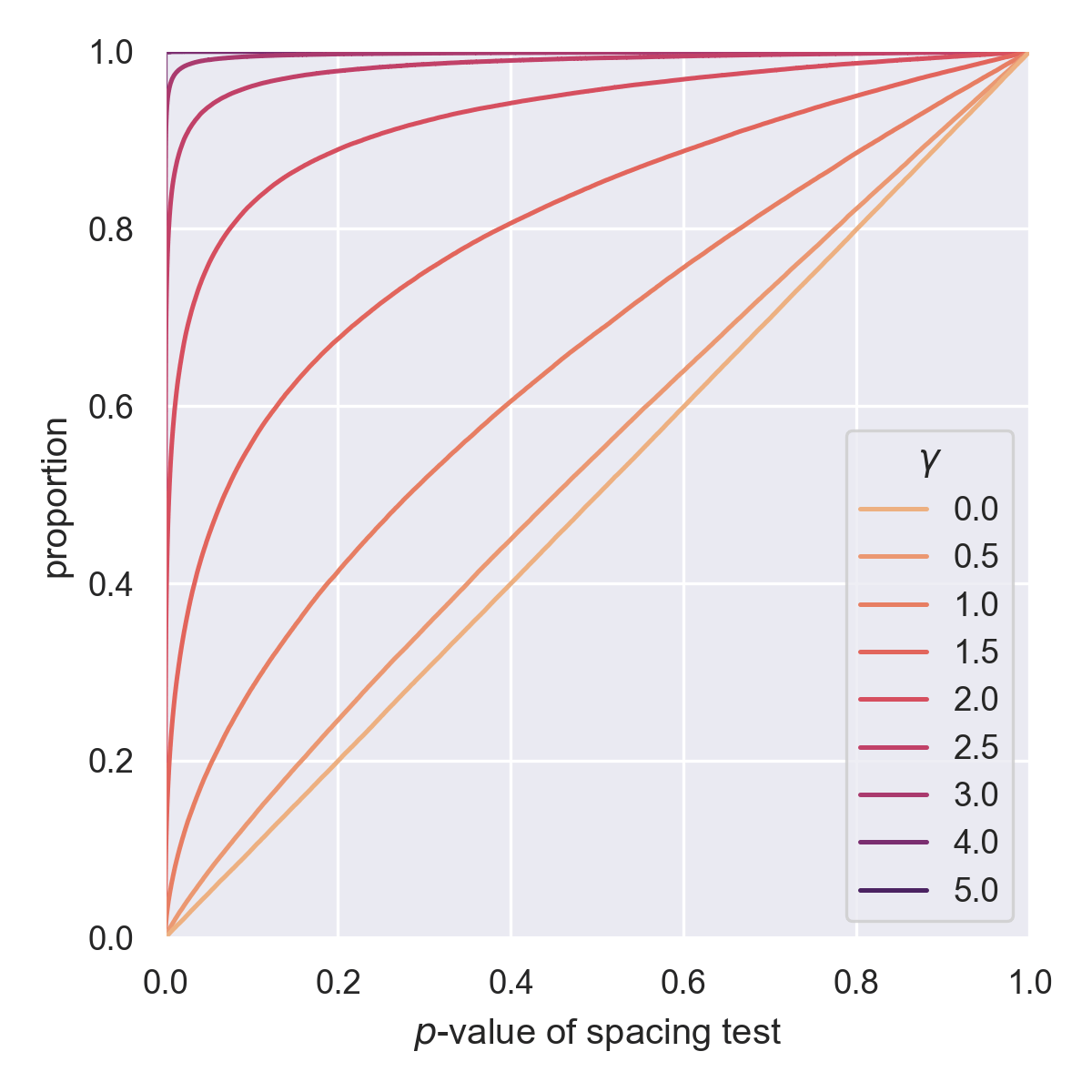}
\hspace*{1.5cm}
\includegraphics[height=\sizespec]{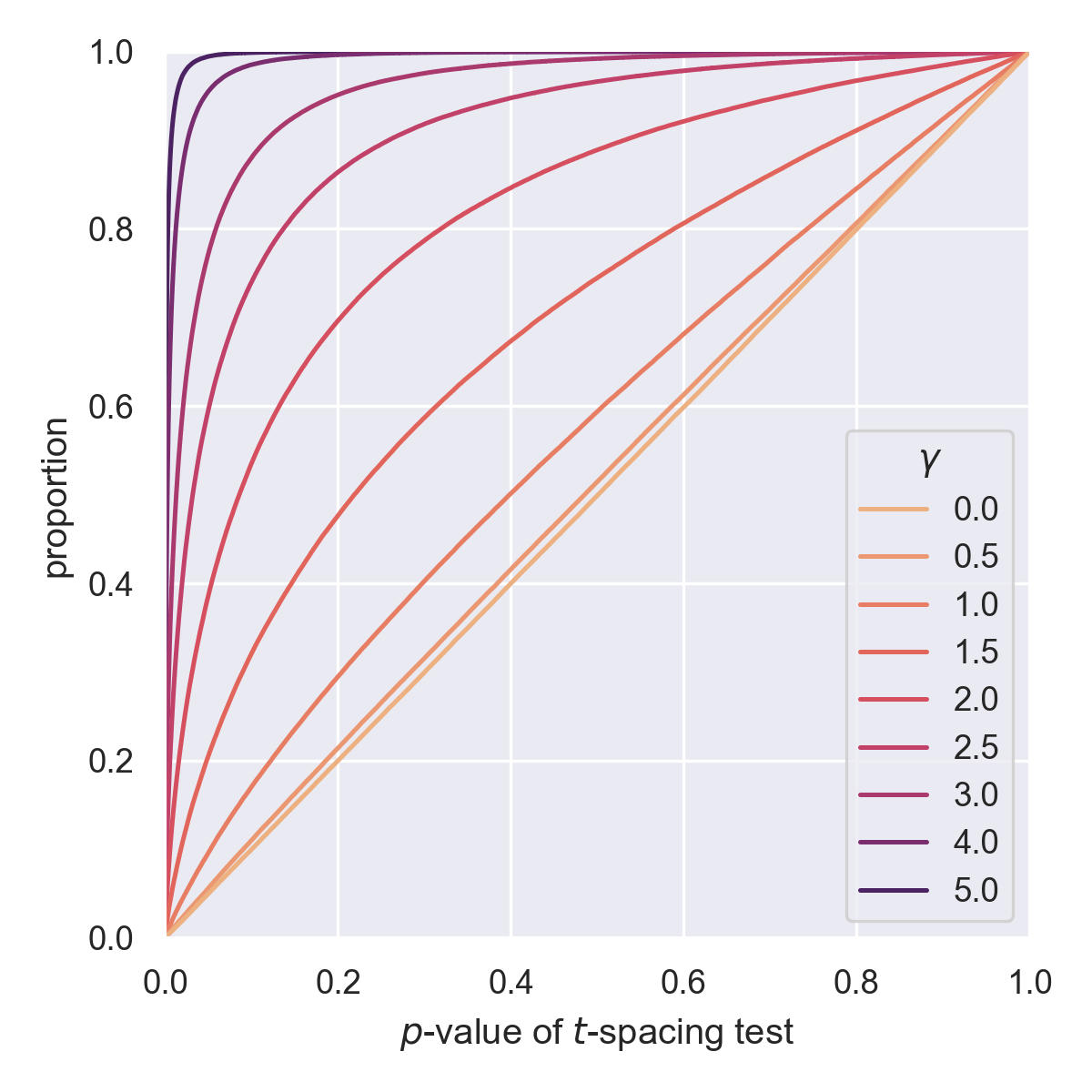}
\caption{{\bf [Spiked tensor PCA, Section~\ref{sec:spike_tensor_model}, example 2/5]} The CDFs of the $p$-value of ($t$-)spacing tests over $250,000$ Monte-Carlo samples for each value of $\gamma$. Note that these tests are perfectly calibrated under the null hypothesis, for which $\gamma=0$. The parameter $\gamma$ is a scaling factor of eigenvalue of the rank-one tensor to be detected. The value~$\gamma=1$ corresponds to the so-called phase transition in Spiked tensor PCA as presented in \citet[Theorem~1.3]{perry2020statistical}. In the $t$-spacing test, the variance $\sigma$ has been estimated on $X^{|t_1}(\cdot)$.}
\label{fig:cdf_spacings}
\end{figure}

\subsection{Detecting one-sparse alternatives in continuous sparse kernel regression}

Our test is perfectly calibrated for {all} $\mathcal C^2$-mean $m(\cdot)$. But, we expect this test to have high power on $s$-sparse alternatives of the form 
\[
m(\cdot)=\sum_{k=1}^{s}\lambda_{0,k} c(\cdot,t_{0,k})\,,
\]
and we denote one-sparse alternatives by, for some $\lambda_0\neq 0$ and $t_0\in  M$ unknown, 
\begin{equation}
\label{eq:def_alternative}
\mathds H_1(t_0,\lambda_0)\,:\,\text{‘‘}\,m(\cdot)=\lambda_0 c(\cdot,t_0)\,\text{''}\,.
\end{equation}

\renewcommand\sizespec{4.5cm}
\begin{figure}[!t]
\begin{centering}
\begin{tabular}{cc}
\includegraphics[height=\sizespec]{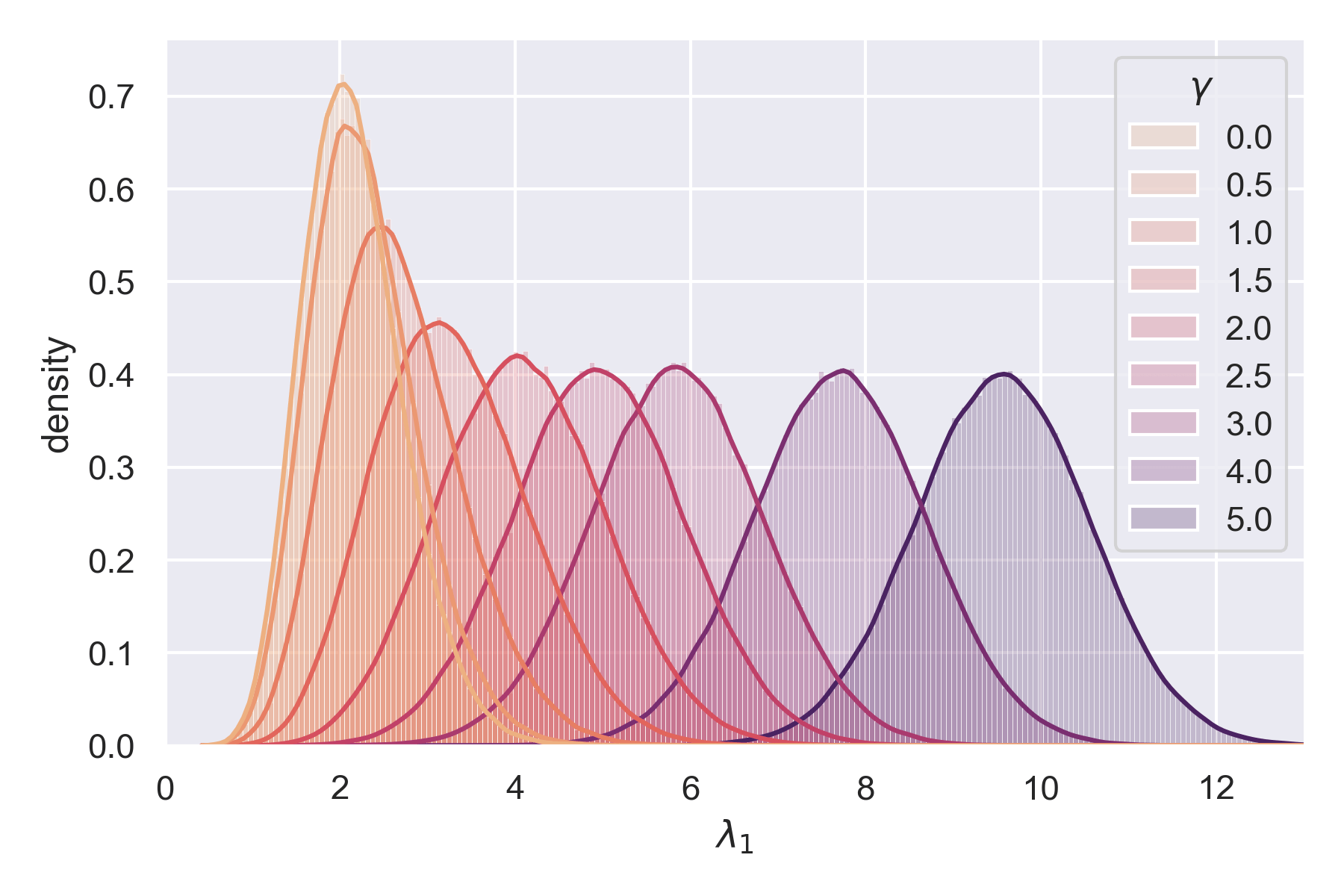}&
\includegraphics[height=\sizespec]{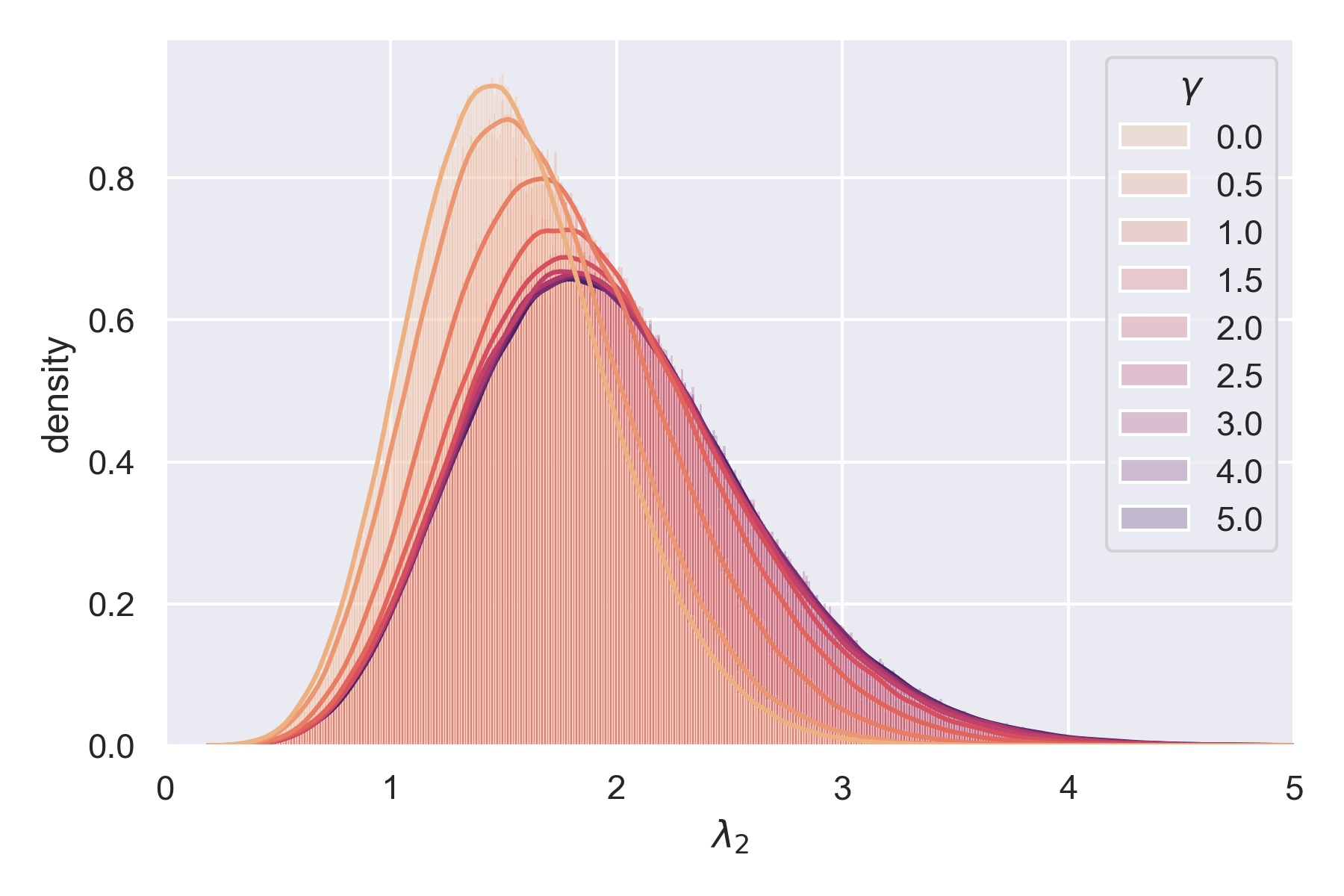}\\
\includegraphics[height=\sizespec]{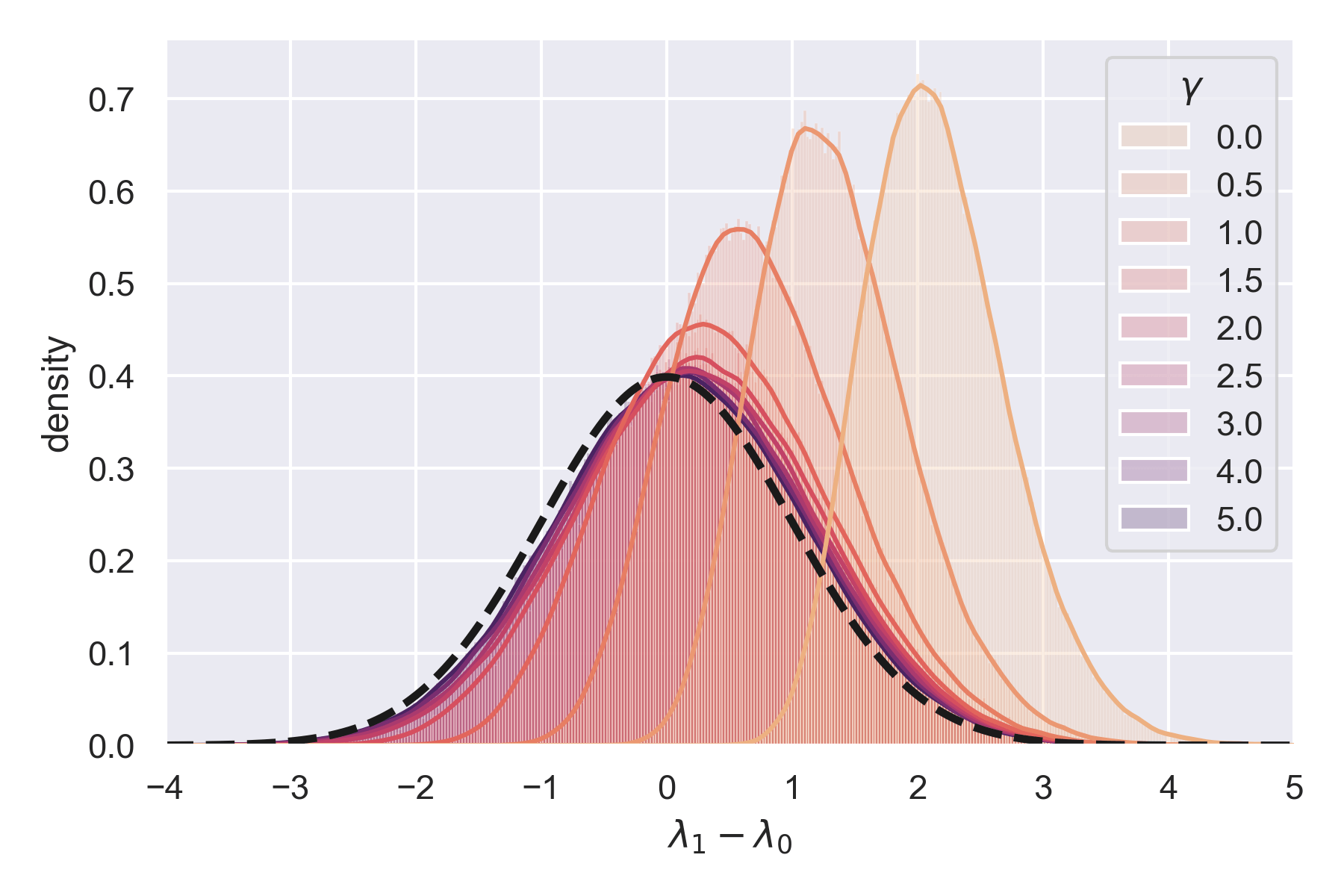}&
\includegraphics[height=\sizespec]{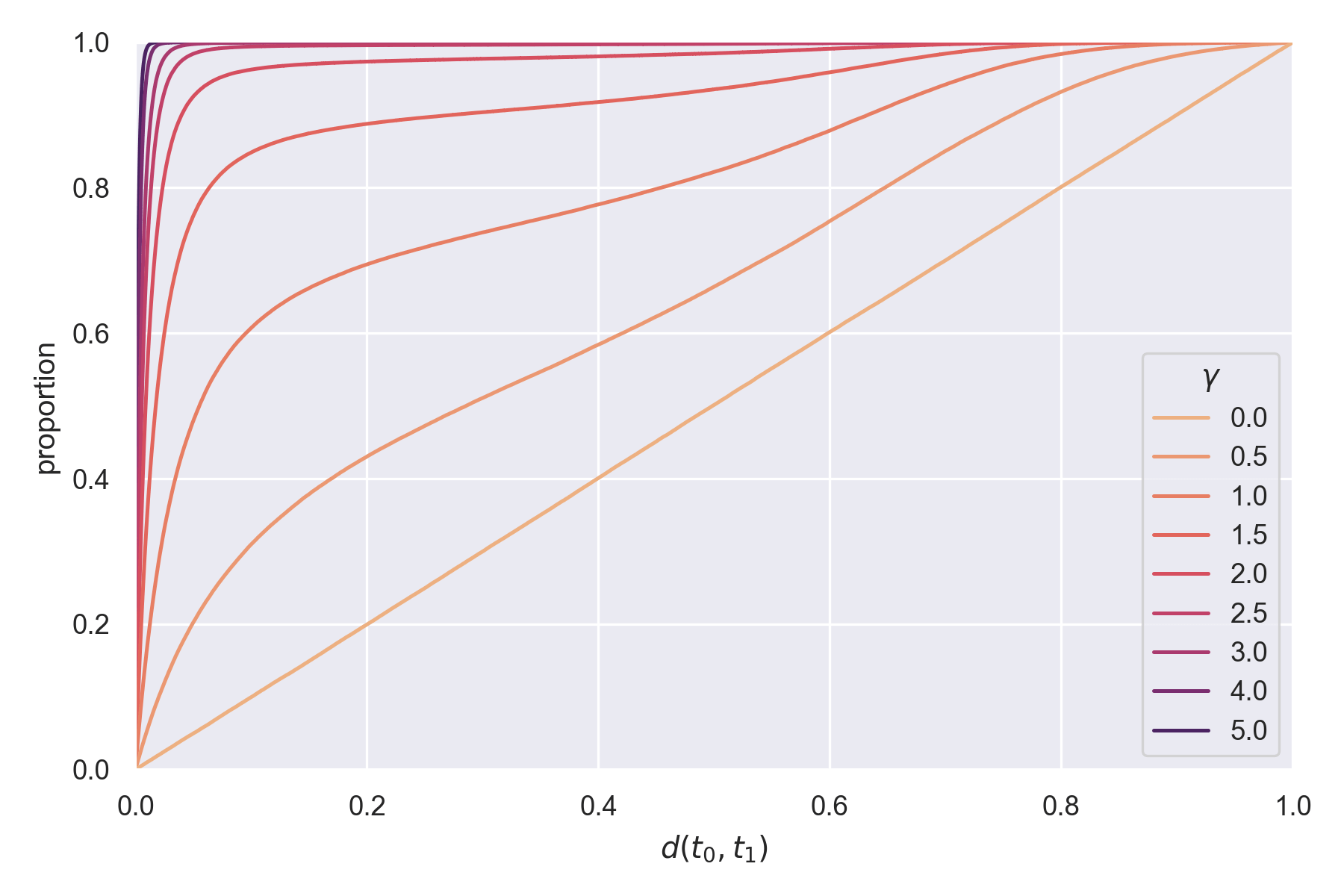}
\end{tabular}
\end{centering}
\caption{{\bf [Spiked tensor PCA, Section~\ref{sec:spike_tensor_model}, example 3/5]}
The PDFs of the maxima $\lambda_1,\lambda_2$ and the CDF of the distance $\mathrm{d}(t_0,t_1)$ over $250,000$ Monte-Carlo samples for each value of the parameter $\gamma$. The alternative is given by $t_0$ fixed and $\lambda_0=\gamma \times\sigma \sqrt{3\log 3+3\log\log 3}\simeq 0.684 \gamma\sigma$ where $\gamma=1$ corresponds to the so-called phase transition in Spiked tensor PCA as presented in \citet[Theorem 1.3]{perry2020statistical}. The distance~$\mathrm{d}(t_0,t_1)$ is normalized so that it is uniformly distributed on $(0,1)$ if $t_1$ is uniformly distributed on the sphere (e.g., $\gamma=0$).
}
\label{fig:lamddas}
\end{figure}

Indeed, as detailed in Section~\ref{sec:lars}, the maxima $\lambda_1$ and $\lambda_2$ correspond to \emph{the knots of the Continuous LARS algorithm and are intrinsic to continuous kernel sparse regression}. We illustrate this behavior in Figure~\ref{fig:lamddas} using the Spiked Tensor PCA model. In this experiment, the alternative is parameterized by~$\lambda_0$, where $\gamma$ scales the signal-to-noise ratio.\footnote{The signal is $\lambda_0=\gamma \times\sigma\sqrt{3\log 3+3\log\log 3}$. Here, $\gamma=1$ corresponds to the phase transition threshold derived in \citet[Theorem~1.3]{perry2020statistical}. Note that in their framework, the threshold is $\sqrt{2\log 3+2\log\log 3}$; we use a different normalization of the Gaussian noise, resulting in a statistical threshold of $\sigma\sqrt{d\log k+d\log\log k}$. Our Gaussian noise is presented in~\eqref{e:noise_tensor} and differs by $\sigma\sqrt{d/2}$ from theirs (see~\cite[Equation~10]{montanari2015limitation}), where~$\sigma$ is an unknown noise level.} As illustrated in the top panels, $\lambda_1$ stochastically increases with $\gamma$, while the distribution of $\lambda_2$ remains stable for moderate signal (e.g., $\gamma \geq 2$). Consequently, the spacing $\lambda_1 - \lambda_2$ grows linearly with $\gamma$. In this moderate regime, the bottom panels confirm that the maximizer $t_1$ concentrates around the true parameter $t_0$ ($t_1 \simeq t_0$) and that $\lambda_1$ follows the distribution of $\sigma X(t_0)$ under $\mathds H_1(t_0,\lambda_0)$ (Gaussian with mean $\lambda_0$ and variance $\sigma^2$). This implies that $(t_1, \lambda_1)$ weakly converges to the distribution of $(t_0, \sigma X(t_0))$, allowing the ($t$-)spacing tests to successfully detect the alternative. We leave the formal proof of this specific phenomenon for future work.

\medskip

\noindent\textbf{Comparison with asymptotic phase transitions.} 
It is crucial to distinguish our contributions from well-established statistical limits in Spiked Tensor PCA. The existing literature, such as \citet[Theorem~1.3]{perry2020statistical}, proves that detecting a rank-one tensor below a specific eigenvalue threshold is impossible in the asymptotic regime where the dimension $d \to \infty$. In contrast, the ($t$-)spacing test proposed here provides an \textit{exact, non-asymptotic} inference procedure valid for any fixed dimension $d$ and sample size. While our numerical experiments (Figure~\ref{fig:lamddas}) recover the phenomenological behavior of this phase transition, where power drops to zero below the theoretical threshold, our procedure guarantees perfect calibration under the null hypothesis regardless of signal strength or dimension. Furthermore, unlike theoretical thresholds that typically require knowledge of the noise level $\sigma$, our $t$-spacing test \emph{remains exact even when the noise variance is unknown}. 

To the best of our knowledge, this is the first exact non-asymptotic test for the spiked tensor model with unknown $\sigma$. Finally, while Tensor PCA serves as a primary motivating example, our framework applies to the broader context of Gaussian random fields on Riemannian manifolds and general continuous sparse kernel regression. Future work may explore the precise links between our method and maximum likelihood detection thresholds (see Section~\ref{s:kernel}). The behaviour of the test when $\lambda_1$ and~$\lambda_2$ are close, and the numerical evaluation of the $p$-value, are examined at the end of Section~\ref{sec:spike_tensor_model} (Figure~\ref{fig:spacing_close_maxima}).

\subsection{A general framework related to continuous kernel regression}

\subsubsection{The four conditions on the Gaussian random field}
We start by giving the assumptions of the Gaussian random field $X(\cdot)$ that we will invoke along this paper. We consider a real valued Gaussian random field~$X(\cdot)$ defined on a $\mathcal C^2$ compact Riemannian manifold~${ M}$ of dimension $d$ without boundary.  We recall that $c(\cdot,\cdot): { M}\times { M}\to \R$ denotes its variance-covariance function. 

\medskip

\noindent
{\bf Assumption $\bf (A_{1\text{-}4})$} We assume that 
\begin{align}
\label{c0}
\circ\ 
&\text{the \textit{paths of $X(\cdot)$ are $\mathcal C^2$ almost surely}};
\tag{$\bf A_1$}\\
\label{c1}
\circ\ &\E [X(t)] =0\text{ and }\Var[X(t)]=1\text{, for every }t\in { M};
\tag{$\bf A_2$}
\\
 \label{c2}
 \circ\ &\forall s\neq t\,,\ c(s,t)<1;
\tag{$\bf A_3$}\\
 \label{c3}
\circ\ &\text{for every }t\in { M} ,\text{ the gradient } \nabla X(t) \text{ has a non-degenerate } \tag{$\bf A_4$}\\
\notag &\text{Gaussian distribution}\,.
\end{align}
The variance-covariance matrix of the Gaussian tangent vector $ \nabla X(t) $ is denoted by $\Lambda_2(t)$, which is invertible by Assumption \eqref{c3}. 

\subsubsection{A continuous regression framework for the spacing test}
\label{s:kernel}
\subsubsection*{General framework}
We consider the following regression problem given by %
\begin{equation}
\label{eq:def_Y_regression}
	Y=\lambda_0\psi_{t_0}+\sigma W\,,
\end{equation}
where $\lambda_0\in\bbR$, $t_0\in\mathrm M$ are unknown parameters; $\sigma$ is the standard deviation; $(E,\langle\cdot,\cdot\rangle_E)$ is any Euclidean space; $\psi\,:\,t\in M\mapsto \psi_t\in E$; $W\in E$ is a centered Gaussian vector with variance-covariance matrix~$\mathrm{Id}_E$. The {\it feature map} $\psi$ satisfies the following Assumptions $\bf (B_{1\text{-}6})$:
\begin{align}
\label{d1}
\circ\ 
&\psi\text{ is a $\mathcal C^{3}$-function};
\tag{$\bf B_1$}\\
 \label{d3}
 \circ\ &\forall t\in M\,,\ \|\psi_t\|_E^2=1;
\tag{$\bf B_2$}\\
 \label{d4}
 \circ\ &\forall t\in M\,,\ \forall s\in M\setminus\{t\}\,,\ \langle\psi_s,\psi_t\rangle_E<1;
\tag{$\bf B_3$}\\
 \label{d6}
 \circ\ &\forall t\in M\,,\ \text{$J\psi_t{J\psi_t}^\top$ has full rank $d$};
\tag{$\bf B_4$}\\
 \label{d5}
 \circ\ &\text{the span of $\{\psi_t\ :\ t\in M\}$ has dimension $m$ such that $m>d+1$};
\tag{$\bf B_5$}\\
\label{d2}
\circ\ &
\forall t\in M\,,\ \exists s\in M\ \text{s.t.}\ \psi_t=-\psi_s;
\tag{$\bf B_6$}
\end{align}
where $J\psi_t$ is the Jacobian matrix of $\psi_t$ and ${J\psi_t}^\top$ its transpose. In the next paragraph, we will see that $X(t)=\langle W,\psi_t\rangle_E$ and 
\[
\forall s,t\in M\,,\quad c(s,t)=\langle \psi_s,\psi_t \rangle_E\,.
\]
One can check that \eqref{d1} implies \eqref{c0}, \eqref{d3} is equivalent to \eqref{c1}, \eqref{d4} is equivalent to~\eqref{c2}, \eqref{d6} is equivalent to \eqref{c3}, and \eqref{d5} is equivalent to the forthcoming Assumption~\eqref{c4}. 

\begin{remark}[Spiked tensor PCA, Section~\ref{sec:spike_tensor_model}]
In Figures~\ref{fig:eigenvector}, \ref{fig:cdf_spacings} and \ref{fig:lamddas}, we considered, as an example, the detection of a rank one $3$-way symmetric tensor. We consider the sphere $\mathrm M= \mathds S^2$. The Euclidean space of $3$-way symmetric tensors of size $3\times3\times3$ is denoted by $(E,\langle\cdot,\cdot\rangle_E)$ with dimension $m=10$, and
\begin{align*}
\psi_t &:= t^{\otimes 3} =(t_i t_j t_k)_{1\leq i,j,k\leq 3\,,}\\
X(t) &:= \langle W,\psi_t\rangle_E=\sum_{ijk} t_i t_j t_k W_{ijk}\,,
\end{align*}
where $ W $ is an order $3$ symmetric tensor defined using symmetry and the following independent terms,
\begin{align*}
\circ\ &W_{iii} \text{ of variance } 1; \text{ there are } 3 \text{ principal diagonal terms}, i\in[3]; \\
\circ\ &W_{iij} \text{ of variance } 1/3;\text{ there are } 6 \text{ sub-diagonals terms}, 1\leq i\neq j\leq 3;\\
\circ\ &W_{123} \text{ of variance } 1/6;\text{ there is }  1 \text{ off-diagonal term}\,.
\end{align*}
\end{remark}

\subsubsection*{Maximum likelihood estimator (MLE)}
The Maximum Likelihood Estimator (MLE) of the couple $(\lambda_0,t_0)$ is given by 
\[
\argmin_{(\lambda,t)}\Big\{\frac12\big\|Y-\lambda\psi_t\big\|_E^2\Big\}\,.
\]
Minimization with respect to $\lambda$ is straightforward under Assumptions \eqref{d3} and \eqref{d2} leading~to 
\begin{equation}
\label{eq:Kernel_Regression_MLE}
\argmin_{t}\min_\lambda\Big\{\frac12\big\|Y-\lambda\psi_t\big\|_E^2\Big\}=\argmax_{t}\,\langle Y,\psi_t\rangle_E\,.
\end{equation}
Now consider the Gaussian random field $Z(t):=\langle Y,\psi_t\rangle_E$, referred to as the profile likelihood. The argument maximum (resp. the maximum) of $Z(\cdot)$ is exactly the MLE of $t_0$ (resp. of~$\lambda_0$) by~\eqref{eq:Kernel_Regression_MLE}, namely
\[
\lambda^{\mathrm{MLE}}=\max Z\quad\text{and}\quad t^{\mathrm{MLE}}=\argmax Z\,.
\]
By~\eqref{eq:def_Y_regression}, note that we uncover the decomposition~\eqref{e:Z_process} and the alternative hypothesis \eqref{eq:def_alternative} as
\begin{equation}
\label{def:observed_rf}
Z(t)=\langle Y,\psi_t\rangle_E 
= \lambda_0\langle \psi_{t_0},\psi_t\rangle_E+\sigma\langle W,\psi_t\rangle_E
= \lambda_0 c(t_0,\cdot) + \sigma X(t)
\,,
\end{equation}
where $X(t):=\langle W,\psi_t\rangle_E$ and $m(\cdot)=\lambda_0 c(t_0,\cdot) $.

\begin{remark}
Assumption~\eqref{d2} ensures that the profile likelihood is given by $Z(\cdot)$, see~\eqref{eq:Kernel_Regression_MLE}. 
\end{remark}

\subsubsection{Parallel with the LARS algorithm}
\label{sec:lars}

\subsubsection*{Discrete LARS}
The Least Angle Regression (LARS) algorithm has been introduced in \citet{efron2004least} and has been widely used in statistics for variable selection. Given an observation $Y\in\mathds R^n$ and $p$ covariates~$(\psi_t)_{t=1}^p$, the LARS algorithm is a forward stepwise variable selection algorithm giving a sequence~$(\lambda_i,t_i)$ of the so-called knots. The first knot $(\lambda_1,t_1)$ is the maximum and the argument maximum of the maximal absolute correlation between the observation and the covariates. Hence, the first step of the LARS aims at maximizing
\begin{subequations}
\begin{equation}
  \label{eq:lambda1_t1_discrete_LARS}
  \lambda_1=\max_{1\leq t\leq p} |\langle Y,\psi_t\rangle_{\mathds R^n}|
  \quad
  \text{and}
  \quad
  t_1=\arg\max_{1\leq t\leq p} |\langle Y,\psi_t\rangle_{\mathds R^n}|\,.
  \end{equation}
We define the so-called residual by, for all $\lambda\leq\lambda_1$, 
  \begin{equation}
    \forall t\in\{1,\ldots,p\}\,,\quad
    Z^{\lambda}(t):=\langle Y-(\lambda_1-{\lambda})\psi_{t_1},\psi_t\rangle_{\mathds R^n}\,,
  \end{equation}
The second knot is defined by 
  \begin{equation}
    \lambda_1-\lambda_2:=\inf\big\{\varepsilon>0\ :\ \exists t\neq t_1\text{ s.t. }Z^{\lambda_1-\varepsilon}(t)\geq\lambda_1-\varepsilon\big\}\,.
  \end{equation}
\end{subequations}
The second knot $(\lambda_2,t_2)$ is built so that $t_1$ is the unique argument maximum of $Z^{\lambda}(\cdot)$ for $\lambda_2<\lambda\leq\lambda_1$ and that there exists a second point, say $t_2\neq t_1$, such that $Z^{\lambda_2}(t_1)=Z^{\lambda_2}(t_2)=\lambda_2$. 

\subsubsection*{Continuous LARS}
The continuous LARS has been investigated in \citet{azais2020testing} for continuous sparse regression from Fourier measurements, see Section~\ref{sec:SR} for further details. We introduce the continuous LARS in a general setting as follows. The first knot $(\lambda_1,t_1)$ is the maximum and the argument maximum of the maximal absolute correlation between the observation $Y\in E$ and $\psi_t\in E$, the features. Hence, the first step of the LARS aims at maximizing 
\begin{subequations}
\begin{equation}
\label{eq:lambda1_regression}
\lambda_1=\max_{t\in { M}} |\langle Y,\psi_t\rangle_E|
\quad
\text{and}
\quad
t_1=\arg\max_{t\in { M}} |\langle Y,\psi_t\rangle_E|\,.
\end{equation}
Under \eqref{d2}, note that $\max_{t\in { M}} |Z(t)|=\max_{t\in { M}} Z(t)$,  where $Z(t)=\langle Y,\psi_t\rangle_E$. Therefore, the definition of~$\lambda_1$ given in \eqref{eq:lambda1_t1_intro} coincides with the aforementioned definition of $\lambda_1$. One can choose~$t_1$ being the unique argument maximum of $Z(\cdot)$, by \eqref{d1} and \eqref{d4} (see also \eqref{lem:tsirelson} in Lemma~\ref{lem:cov}). We define the so-called residual by, for all $\lambda\leq\lambda_1$,
\begin{equation}
Z^{\lambda}(t):=\langle Y-(\lambda_1-{\lambda})\psi_{t_1},\psi_t\rangle_E=Z(t)-(\lambda_1-{\lambda})c(t,t_1)\,,
\end{equation}
where $c(s,t)=\langle \psi_s,\psi_t\rangle_E$. Again, under~\eqref{d2}, note that $\max_{t\in { M}} |Z^\lambda(t)|=\max_{t\in { M}} Z^\lambda(t)$. Moreover, one can check that $Z^{\lambda}(t_1)=\lambda$. The second knot is defined by 
\begin{equation}
\lambda_1-\lambda_2:=\inf\big\{\varepsilon>0\ :\ \exists t\neq t_1\text{ s.t. }Z^{\lambda_1-\varepsilon}(t)\geq\lambda_1-\varepsilon\big\}\,.
\end{equation}
The second knot $(\lambda_2,t_2)$ is built so that $t_1$ is the unique argument maximum of $Z^{\lambda}(\cdot)$ for $\lambda_2<\lambda\leq\lambda_1$ and that there exists a second point, say $t_2\neq t_1$, such that $Z^{\lambda_2}(t_1)=Z^{\lambda_2}(t_2)=\lambda_2$. Now observe that, for all $t\neq t_1\in { M}$ and all $\lambda$,
\begin{align*}
Z^{\lambda}(t)\leq\lambda 
\ &\Leftrightarrow\ %
Z(t) - \lambda_1 c(t,t_1)\leq \lambda(1- c(t,t_1))
\\ &\Leftrightarrow\ %
\frac{Z(t) - Z(t_1) c(t,t_1)}{1- c(t,t_1)}\leq\lambda
\\ &\Leftrightarrow\ %
Z^{|t_1}(t)\leq \lambda\,,
\end{align*}
where $Z^{|t_1}(t)$ is defined as in \eqref{e:Xbarra_intro},
taking into account that the gradient is vanishing at point $t_1$. We uncover that $(\lambda_2,t_2)$ is the second knot of the LARS where 
\begin{equation}
\label{eq:lambda2_regression}
\lambda_2=\max_{t\in { M}\setminus\{t_1\}} Z^{|t_1}(t)
\quad
\text{and}
\quad
t_2=\argmax_{t\in { M}\setminus\{t_1\}} Z^{|t_1}(t)\,.
\end{equation}
\end{subequations}
Note that the definition of~$\lambda_2$ in \eqref{e:Xbarra_intro} is equivalent to the aforementioned definition of $\lambda_2$ (this is also true for the argument maximum). This paper gives the distribution of $\lambda_1$ (first knot of the continuous LARS) conditional on $(\lambda_2,\Omega)$ (second knot of the continuous LARS and independent part of the Hessian at the first knot) under the null hypothesis.

\subsection{Contributions and outline}

\subsubsection*{Main contributions.} We investigate the \emph{second maximum} $\lambda_2$ of a $\mathcal C^2$ centered Gaussian random field on a compact manifold $M$, a quantity naturally arising along the continuous LARS path of sparse kernel regression. Classical inference relying only on the global maximum $\lambda_1$ typically uses tail or Euler characteristic approximations; here, by conditioning on the richer object $(\lambda_2,\Omega)$---where $\Omega$ denotes the independent part of the Hessian at the (random) maximiser $t_1$---we obtain \emph{exact, non-asymptotic} distributional identities for $\lambda_1$ and devise spacing tests. Our main results are:
\begin{itemize}
  \item[(C1)] \textbf{Geometric characterisation of $\lambda_2$.} Definition of $\lambda_2$ as the maximum of the singular conditional field $X^{|t_1}(\cdot)$ with a removable first-order structure, and study of its basic properties (Section~\ref{s:second}).
  \item[(C2)] \textbf{Helix / eigen-structure link.} A correspondence (Lemma~\ref{lem:helix}) between the eigen-decomposition of $\Omega$ and directional limiting values of the rescaled field near $t_1$, clarifying local geometry around the global maximum.
  \item[(C3)] \textbf{Singular Kac--Rice formula.} An ad-hoc Kac--Rice formula adapted to $X^{|t_1}(\cdot)$ yielding the explicit conditional density of $\lambda_1$ given $(\lambda_2,\Omega)$ (Theorem~\ref{thm:joint_law}, Appendix~\ref{app:weights}).
  \item[(C4)] \textbf{Exact spacing test (known variance).} A finite-sample valid test based on the spacing $\lambda_1-\lambda_2$ (Theorem~\ref{thm:rice_known_variance}).
  \item[(C5)] \textbf{Variance estimation.} Construction of a Karhunen--Lo\`eve based estimator $\hat\sigma$ for incorporation into the Kac--Rice framework (Section~\ref{sec:tst}).
  \item[(C6)] \textbf{Studentised spacing test.} Exact inference using the studentised spacing $\lambda_1/\hat\sigma-\lambda_2/\hat\sigma$ without prior knowledge of~$\sigma^2$ (Theorem~\ref{thm:rice_unknown_variance}, Proposition~\ref{prop:Rice_density_unknown_variance}).
  \item[(C7)] \textbf{Empirical validation and applications.} Simulations confirm calibration and improved power for sparse detection (\url{https://github.com/ydecastro/tensor-spacing/}); links with continuous sparse kernel regression and super-resolution are detailed in Section~\ref{sec:examples}.
\end{itemize}
Conceptually, the work extends Kac--Rice formula to a controlled singular setting and shows that conditioning on $(\lambda_2,\Omega)$ unlocks exact finite-sample inference for $\lambda_1$. Statistically, spacing-based tests can be studentized leading to exact inference without prior knowledge of $\sigma^2$. Geometrically, the helix interpretation ties curvature encoded in $\Omega$ to directional approach of the conditional~field.

\medskip

These results are also of interest from a purely probabilistic viewpoint. The weight $\det(\ell\,\mathrm{Id}-\Omega)$ that appears in our conditional law is the same determinant weight as in the Kac-Rice and expected Euler characteristic formulas for the level sets of a Gaussian field. Our identities give an exact finite-sample version of these classical computations. The spacing $\lambda_1-\lambda_2$ measures how isolated the highest peak is, and Lemma~\ref{lem:helix} describes the local geometry of the field at its maximum. In this sense the paper adds to the study of the geometry of Gaussian fields on Riemannian manifolds \citep{adler2009random,Azais_Wschebor_09,auffinger2013complexity,armentano2023general}.

\subsubsection*{Outline.} Section~\ref{s:second} introduces $\lambda_2$ and the conditional singular field $X^{|t_1}(\cdot)$. Section~\ref{sec:KRF} presents the singular Kac--Rice formula and Theorem~\ref{thm:joint_law}. Section~\ref{sec:test} develops spacing-based tests for known variance. Section~\ref{sec:tst} constructs $\hat\sigma$ and derives the plug-in law leading to the unknown-variance test. Section~\ref{sec:examples} details connections with continuous sparse kernel regression and super-resolution. Appendices collect technical lemmas, weight computations and further numerical illustrations. A table of list of notation is given in Appendix~\ref{app:notation}.

\subsection{Related works}
The first paper to consider the Kac-Rice formula on manifolds, specifically the sphere and the Stiefel Manifold, is \cite{AW05onthedistribution}. This theory has been comprehensively developed in the monographs by \citet{adler2009random}, \citet{Azais_Wschebor_09} and \citet{Armentano2026}.  In high-dimensional settings, complexity results for random smooth functions on the sphere were derived by \citet{auffinger2013complexity}. A new set of hypotheses and proofs for the Kac-Rice formula, particularly for the measure of level sets, is given in \citet{armentano2023general}. 

The Kac-Rice formula is also central to \citet{azais2017rice} for studying the maximum of Gaussian fields on the torus, with applications to the Super-Resolution problem. This falls under the general theory of continuous sparse regression over the space of measures, which has attracted significant attention in signal processing \citep{candes2014towards,duval2015exact,azais2015spike}, machine learning  \citep{de2021supermix}, and optimization \citep{chizat2022sparse}. The super-resolution framework aims to recover fine-scale details from low-frequency measurements and has applications in astronomy, medical imaging, and microscopy. The novel aspects of this body of work rely on new statistical and optimization guarantees for sparse regression. Initiated by the work presented in  \citep{azais2020testing}, we investigate the possibility of detecting a sparse object from a test on the mean~$m(\cdot)$ of a Gaussian random field under a sparsity assumption. Complementary to these exact results, the Euler Characteristic (EC) heuristic provides an approximation for excursion probabilities and the distribution of the maximum~\citep{Adler2008}, and smoothing techniques can be used on the sphere to detect the location of local maxima as in~\cite{18-BEJ1068}.

Our construction is also related to the selective (post-selection) inference literature, and in particular to the Kac-Rice test of \citet{taylor2016}; see also \citep{taylor2014exact,lockhart2014significance,tian2018selective}. In that line of work, an exact $p$-value for a global null is obtained by conditioning the first knot of a solution path on the second knot, which yields a truncated-Gaussian ratio that is uniform under the null. %
In the finite-dimensional setting, the same pivot can also be reached by direct arguments on the knots of the LARS path, and it can be studentized to allow for an unknown variance, see \citet{lockhart2014significance}, \citet{azais2018power} and \citet{azais2022multiple}; this line of work is the starting point of the present paper.

Although the two pivots share the same form, the two settings differ in an essential way. \citet{taylor2016} study a Gaussian process with a finite Karhunen-Loève expansion, maximised over a stratified convex set of $\bbR^p$ (polyhedral for the lasso, with curved strata for the group lasso and the nuclear norm). Here the natural underlying regression is over the space of measures on $M$ equipped with the total-variation norm, the Beurling-LASSO \citep{de2012exact,azais2015spike,candes2014towards,duval2015exact}: the dictionary is indexed by the Riemannian manifold $M$, and $\lambda_1,\lambda_2$ coincide with the first two knots of the continuous LARS path introduced in Section~\ref{sec:lars} (see \citealp{azais2020testing} for the identification with the Beurling-LASSO path in the super-resolution case). As in their curved-stratum examples, the maximiser is a random point at which the gradient vanishes; but here the conditional field $X^{|t_1}$ is singular at $t_1$ (the helix of Lemma~\ref{lem:helix}), a situation that the positive-reach assumption of \citet{taylor2016} is designed to exclude, so an ad-hoc Kac-Rice argument is needed. As in \citet{taylor2016}, whose pivot carries a determinant weight built from a Hessian term (denoted $\Lambda_{\eta^*}$ there), our pivot involves the independent part of the Hessian $\Omega$; it is their Kac-Rice viewpoint that carries over to this new setting. Theorems~\ref{thm:joint_law} and~\ref{thm:rice_known_variance} hold for any $\mathcal C^2$ Gaussian field satisfying Assumption $\bf (A_{1\text{-}4})$, without requiring a finite Karhunen-Loève expansion, and the $t$-spacing test remains exact when the variance is unknown. For matrix PCA, the pivot of \citet{taylor2016} is a spacing test between the two largest singular values; for the order-$3$ spiked tensor model of Section~\ref{sec:spike_tensor_model}, no such linear-algebraic expression of $\lambda_2$ exists and the singular conditional field must be handled directly. Our contribution is threefold. We handle a continuous Gaussian field on a manifold and describe its local geometry at $t_1$ through the independent part of the Hessian $\Omega$. We obtain the exact conditional law of $\lambda_1$ given $(\lambda_2,\Omega)$ by an ad-hoc singular Kac-Rice formula. And we build an exact test that stays valid when the variance $\sigma^2$ is unknown, which we call the $t$-spacing test (in the finite-dimensional case, exact selective tests with unknown variance are given in \citet{tian2018selective}). The spiked tensor PCA model of Section~\ref{sec:spike_tensor_model} is our main application.

\section{The second maximum of a Gaussian random field} 
\label{s:second}

\subsection{The two Gaussian regression remainders}
\label{sec:regression_remainders}

Some useful properties are presented in the next lemma. 
\begin{subequations}
\begin{lemma}
\label{lem:cov}
Under Assumption $\bf (A_{1\text{-}4})$, one has 
\begin{align}
\label{eq:covariance}
  \circ\ & c(s,t) = \E[X(s)\, X(t)]\text{ and } c(t,t)=1;\\
\label{e:ind_rf_grad}
  \circ\ &X(t) 
\text{ and }
\nabla X(t)
\text{ are independent;} \\
\label{d:lambda2-tilde}
   \circ\ &\Lambda_2(t) := \Var[\nabla X(t)]  = -\Cov [\nabla^2X(t),X(t)]\\
\label{lem:tsirelson}
  \circ\ &\text{the argument maximum of }X(\cdot)\text{ is unique.}
\end{align}
\end{lemma}
\end{subequations}
\begin{proof}
The first property \eqref{eq:covariance} is equivalent to \eqref{c1}. The other properties can be deduced by differentiating $c(t,t)=1$. The last statement is a consequence of Tsirelson's theorem, see for instance \citet[Theorem~3]{lifshits1983absolute}. 
\end{proof}

\subsubsection*{The Gaussian random field $X^{|t}(\cdot)$}
We know that $X(t)$ and~$\nabla X(t)$ are independent by \eqref{e:ind_rf_grad}. For a fixed point $t\in { M}$, consider the remainder of Gaussian regression of $X(s)$ with respect to $(X(t),\nabla X(t))$ given by the Gaussian random field 
\[
s\in { M}\mapsto X(s)
-c(s,t)\,X(t)
-\nabla_t c(s,t)^\top\Lambda_2^{-1}(t)\nabla X(t)\in\bbR\,,
\] 
where $\nabla_t c(s,t)$ is the Riemannian gradient of $t\mapsto c(s,t)$. By \eqref{c3}, remark that~$\Lambda_2(t)$ is invertible. This Gaussian random field is well defined on ${ M}$ and independent of $(X(t),\nabla X(t))$. Now, for $s\neq t$, set
\begin{equation} 
\label{e:Xbarra}
 X^{|t}(s) := \frac{X(s)-c(s,t)\,X(t)-\nabla_t c(s,t)^\top\Lambda_2^{-1}(t)\nabla X(t)}{1-c(s,t)},
\end{equation}
that is, $X^{|t}(s)$ is a normalisation of the remainder of the regression of $X(s)$ w.r.t. $(X(t),\nabla X(t))$.  

\subsubsection*{The regression of the Hessian $R(t)$}
Now, consider the following regression in the space of Gaussian symmetric matrices
\begin{subequations}
\begin{equation} \label{d:Rtecho}
 \nabla^2X(t) = -\Lambda_2(t)\, X(t) -\Lambda_3 (t) \nabla X(t) + \tilde R(t),
\end{equation}
for some well-defined $3$-way tensor $\Lambda_3(t)$. It will not be necessary to give the explicit expression of $\Lambda_3(t)$ for our purposes, this tensor is well defined by Gaussian regression formulas. Thus, one can identify the symmetric matrix $\tilde R(t)$ as the remainder of the regression of $\nabla^2X(t)$ on $(X(t),\nabla X(t))$.  For future use, it is convenient to set 
\begin{equation} 
	\label{d:R}
  	R(t):=  \Lambda^{-\frac12}_2(t)\,\tilde R(t)\,\Lambda^{-\frac12}_2(t)\,.
\end{equation} 
Note that $ \tilde R(t)$ and $R(t)$ are symmetric. The following lemma is straightforward.
\end{subequations}
\begin{lemma}
\label{lem:independent}
For any $t\in { M}$, the Gaussian random field $X^{|t}(\cdot)$ and the Gaussian random matrix~$R(t)$ are independent of $(X(t),\nabla X(t))$. 
\end{lemma}

\subsection{Second maximum and independent part of the Hessian}
\noindent
{\bf First maximum:} We define the first maximum $\lambda_1$ of $\sigma X(\cdot)$ and its argument maximum $ t_1$ by 
\begin{equation}
\label{def:hatz}
\lambda_1 := \sigma X({t_1}) \quad\mathrm{and }\quad {t_1}:= \argmax_{t\in { M}} X(t).
\end{equation}
The argument maximum is almost surely a singleton by \eqref{lem:tsirelson}, hence $ t_1$ is unique almost surely. 

\medskip

\noindent
{\bf Second maximum:} We define the second maximum $\lambda_2$ of $\sigma X(\cdot)$ by
\begin{equation} 
\label{eq:lambdaT}
\forall t\in { M},\quad
\lambda_2^t:=\sup_{s\in { M}\setminus\{ t\}}\{\sigma X^{|t}(s)\}\quad \mathrm{and} \quad \lambda_2 := \lambda_2^{ {t_1}}\,.
\end{equation}

\medskip

\noindent
{\bf Independent part of the Hessian:} We define the independent part of the Hessian $\Omega$ as 
\eq
\label{def:R}
\Omega:= \sigma R({{t_1}})= \sigma \Lambda^{-\frac12}_2({{t_1}})\,\tilde R({{t_1}})\,\Lambda^{-\frac12}_2({{t_1}})\,.
\qe

\noindent
At this stage, it is not clear that $\lambda_2^t<\infty$ a.s. and how $X^{|t}(s)$ is shaped around point $t$. The next lemma gives a description of $s\mapsto X^{|t}(s)$ around point $t$ which proves that $\lambda_2^t<\infty$ almost surely. A proof of this lemma can be found in Appendix~\ref{a:helix}. Note that, since ${ M}$  is compact, by the Hopf–Rinow theorem, ${ M}$ is geodesically complete and the exponential map exists on the whole tangent space.

\begin{subequations}
\begin{lemma}
\label{lem:helix}
Let $h$ be a nonzero vector of the tangent space at $t$. For all $\varepsilon\neq 0$, let $s(\varepsilon):=\exp_t(\varepsilon h)\in { M}$ be the exponential map at $t$ given by the tangent vector $\varepsilon h$. Then, under \eqref{c0}, \eqref{c1} and~\eqref{c3}, 
\begin{equation}
\label{e:helix}
\lim_{\substack{\varepsilon\neq 0\\ \varepsilon\to0}}\, X^{|t}(s(\varepsilon))
=\frac{h^\top\tilde R(t)h}{h^\top\Lambda_2(t)h}
=\frac{h^\top}{\|h\|_2}\,R(t)\,\frac{h}{\|h\|_2}\,,\text{ almost surely}
\,.
\end{equation}
Furthermore, there exists a unit norm tangent vector $h_0$ at point $t$ such that 
\[
\limsup_{s\to t}\,  X^{|t}(s)
=h_0^\top{R(t)}\,h_0
= {\lambda_{\max}(R(t))}<\infty\,,\text{ almost surely}\,.
\]
\end{lemma}

\begin{remark}
\label{rem:helix}
The aforementioned lemma shows that $\lambda_2^t$ varies in $(-\infty ,\infty)$ almost surely. Also, it shows that~$X^{|t}(\cdot)$ is a {\tt helix} random field \citep[Lemma~4.1]{AW05onthedistribution} with pole $t$: the paths of the random field need not extend to a continuous function at the point $t$; however, the paths have radial limits at $t$ and the random field may take the form of a helix around $t$.
The shape of the helix locally around the singularity is described by the eigen-decomposition of the independent part of the Hessian, as shown by \eqref{e:helix}.
More precisely, the radial limits of $X^{|t}$ at $t$ are the Rayleigh quotients of the normalised Hessian $R(t)$: the limit along the tangent direction $h$ is the Rayleigh quotient of $R(t)$ at the direction $\Lambda_2(t)^{1/2}h$, and the two directions coincide when $\Lambda_2(t)$ is proportional to the identity, as in the examples of Section~\ref{sec:examples}. The largest limit is $\lambda_{\max}(R(t))$. At the maximiser $t_1$, where $\Omega=\sigma R(t_1)$, this means that the eigenvalues of $\Omega$ describe the shape of the field along the principal directions at the maximum: a larger eigenvalue corresponds to a steeper approach along that direction.
\end{remark}

\noindent
Besides, for every $t$ such that $\nabla X(t)=0$ it follows that 
\begin{equation}
\label{e:barsansbar}
X^{|t}(s) 
= \frac{X(s)-c(s,t)X(t)}{1-c(s,t)} 
=: X^{t}(s).
\end{equation}
In particular, we have the following identity between random fields $X^{|{t_1}}(\cdot)=X^{{t_1}}(\cdot)$, which may not be Gaussian (due to the random point ${t_1}$). The following lemma is straightforward. 
\end{subequations}

\begin{lemma}
\label{lem:selection_event}
For a fixed $t\in { M}$, consider the following indicator functions:
\begin{itemize}
\item $\imath_1:=\1\big\{ t={t_1}\big\}$;
\item $\imath_2:=\1\big\{ \forall s\in { M}\setminus\{t\}\,,\  X(s)\leq X(t)\big\}$;
\item $\imath_3:=\1\big \{\forall s\in { M}\setminus\{t\}\,,\  X^{t}(s)\leq X(s)\big\}$;
\item $\imath_4:=\1\big\{\nabla X(t)=0\big\}\1\big\{\lambda_2^t\leq X(t)\big\}$;
\end{itemize}
then $\imath_1=\imath_2=\imath_3=\imath_4$ and $-\infty<\lambda_2 \leq \lambda_1<\infty$, almost surely.
\end{lemma}

\subsection{The conditional distribution of the maximum}
\label{sec:KRF}
We have the following key result giving the distribution of $\lambda_1$, defined by \eqref{def:hatz}, conditional on~$(\lambda_2,  \Omega)$, and proven in Appendix~\ref{sec:proof_joint_law}.
\begin{theorem}
\label{thm:joint_law}
Let $X(\cdot)$ be a Gaussian random field satisfying Assumption $\bf (A_{1\text{-}4})$. Then, the distribution $\mathcal D(\lambda_1\,|\, \lambda_2,  \Omega)$ of the maximum $\lambda_1$ conditional on $(\lambda_2,  \Omega)$ has a density with respect the Lebesgue measure~$\leb$ and this conditional density at point $\ell\in\bbR$ is given by
\[
\frac{\mathrm d\mathcal D(\lambda_1\,|\, \lambda_2,  \Omega)}{\mathrm d\leb}(\ell)=\frac{\det(\ell\, \mathrm{Id}-\Omega)  \varphi(\ell/\sigma)}{G_{\Omega}(\lambda_2)}{\1}_{ {\lambda_2}\leq \ell}
\quad\text{with}\quad G_{\Omega}(\lambda_2):=\int_{{\lambda_2}}^{+\infty}\!\!\!\!\!\!\det ( u \mathrm{Id}-\Omega)    \phi(u /\sigma)\, \mathrm{d}u
\,.
\]
\end{theorem}

\begin{remark}
The conditional density of $\lambda_1$ given $(\lambda_2,\Omega)$ is proportional to $\det(\ell\,\mathrm{Id}-\Omega)\,\varphi(\ell/\sigma)$ on the set $\{\ell\ge\lambda_2\}$. The factor $\det(\ell\,\mathrm{Id}-\Omega)$ is the determinant weight that also appears in the Kac-Rice and expected Euler characteristic formulas at level $\ell$. Theorem~\ref{thm:joint_law} therefore links the two top values~$\lambda_1$ and~$\lambda_2$ of the field with the spectrum of the Hessian at the maximiser, in one exact finite-sample identity. It can be read as an exact refinement of the usual Euler characteristic and tail approximations for the law of the maximum. Integrating this density gives the pivot \eqref{e:hat_alpha}: the ratio ${G_{\Omega}(\lambda_1)}/{G_{\Omega}(\lambda_2)}$ is the conditional probability that the maximum exceeds $\lambda_1$, so it is uniform on $(0,1)$ under the null. We call the underlying Kac-Rice formula ad-hoc because the classical weighted formula does not apply here, as the next remark explains.
\end{remark}

\begin{remark}
The classical weighted Kac-Rice formula requires the weight to be a fixed function of the point $t$. Here the weight is $\lambda_2^{t}=\sup_{s\neq t}\sigma X^{|t}(s)$, which is a global functional of the whole field, and the conditional field $X^{|t}$ is singular at $t$. The truncation and monotone-convergence argument that we use to handle this (Appendix~\ref{app:weights}) also applies to other Kac-Rice computations whose weight is such a global functional of the field, for instance moments of the number of critical points weighted by a supremum of the field, or the law of the gap between ordered critical values.
\end{remark}

\section{Spacing test for the mean of a random field with known variance}
\label{sec:test}

\subsection{Testing framework}
\label{sec:testing_framework}
We observe a random field $Z(t) = m(t) + \sigma X(t)$ and we would like to test the global nullity of its mean $m(\cdot)$. We define the statistics 
\begin{subequations}
\begin{align}
\lambda_1&:=\max_{t\in  M}\{Z(t)\}\quad\text{and}\quad t_1 :=\argmax_{t\in  M}\{Z(t)\};\\
Z^{|t_1}(s)&:=\frac{Z(s)-\lambda_1 c(s,t_1)}{1-c(s,t_1)};\\
\lambda_2&:=\max_{s\in  M}\{Z^{|t_1}(s)\}\quad\text{and}\quad t_2:=\argmax_{s\in  M}\{Z^{|t_1}(s)\};\\
\Omega&:=\Lambda^{-\frac12}_2(t_1)\,(\nabla^2Z(t_1) +\lambda_1\Lambda_2(t_1))\,\Lambda^{-\frac12}_2(t_1).
\end{align}
\end{subequations}
In the previous section, we assumed that the Gaussian random field $X(\cdot)$ was centered in~\eqref{c1}. In this section, we give an exact test procedure for the following null hypothesis:
\begin{equation}
\label{e:H0}
\tag{${\mathds H}_0$}
\bbE[Z(\cdot)]=0\,,
\end{equation}
as a consequence $Z(\cdot)=\sigma X(\cdot)$ and the notations $\lambda_1,t_1,\lambda_2,t_2,\Omega$ are consistent with Section~\ref{s:second}. 

\subsection{Spacing test}
\noindent We can now state our main  result when the variance $\sigma^2$ is known. 

\begin{theorem}
 \label{thm:rice_known_variance}
 Let $X(\cdot)$ be a Gaussian random field satisfying Assumption $\bf (A_{1\text{-}4})$. 
Under~$\mathds H_0$\,, the following test statistic satisfies
\[
\mathbf{S}_\sigma(\lambda_1,\lambda_2,\Omega) :=\frac{  G_{\Omega/\sigma} (\lambda_1/\sigma)}{ G_{\Omega/\sigma} (\lambda_2/\sigma)}\sim\mathcal U(0,1)\,,
\]
where $\mathcal U(0,1)$ is the uniform distribution on $(0,1)$, 
\[
\displaystyle G_{\Omega/\sigma}(\ell):=\int_{\ell}^{+\infty}\!\!\!\!\!\!\det ( u \mathrm{Id}-\Omega/\sigma)    \phi(u )\, \mathrm{d}u\,,
\]
and~$\varphi(\cdot)$ is the standard Gaussian density.
 \end{theorem}
 
\noindent \begin{proof}
Without loss of generality, assume that $\sigma=1$. It is well known  that, if a random variable~$Q$ has a {\it continuous} cumulative distribution function $F$,  then $F(Q)\sim\mathcal U(0,1)$. This implies that, conditionally on  $ (\lambda_2,\Omega)=(\ell_2,r)$,  ${G_r (\lambda_1)}/ {G_r (\ell_2)} \sim \mathcal U(0,1)$. Since the conditional distribution does not depend on $(\ell_2,r)$, it is also the non-conditional distribution as claimed.
\end{proof}

\section{The unknown variance case}
\label{sec:unknown} 

\subsection{Existence of non-degenerate systems}
We consider a real valued centered Gaussian random field $X(\cdot)$ defined on $M$ satisfying Assumptions~\eqref{c1}, for the moment. We consider the order~$m$ Karhunen-Lo\`eve expansion in the sense that
\begin{equation}
\label{a:Klm}
 \sigma X (t)= \sum_{i=1}^{m} \zeta_i f_i (t)\mbox{ with } \Var(\zeta_i) = \sigma^2 \mbox{ and}~\forall t\in { M},\ \sum_{i=1}^{m}| f_i(t)|^2 = 1\,,
 \tag{${\mathrm{KL}(m)}$}
\end{equation}
where the equality holds in $L^2$, uniformly in $t$, and $(f_1,\ldots,f_m)$ is a system of non-zero functions orthogonal on~$L^2({ M})$. We say that $X(\cdot)$ satisfies Assumption~\eqref{a:Klm} if it admits an order~$m$ Karhunen-Lo\`eve expansion. Through our analysis, we need to consider the following non-degeneracy Assumption: $X$ is a.s. differentiable, it holds that $m>d+1$ and for all $t_1 \in { M}$
\begin{align*}
&\exists (t_{d+2},\dots,t_m) \in ({ M}\setminus \{t_1\})^{m -d-1}~\text{ pairwise distincts s.t.}  
\\
&(X(t_1),\nabla X(t_1), X(t_{d+2}),\dots,X(t_m))\ \mathrm{is\ non\ degenerate}.
\label{a:nonDegeneratem+}
\tag{${\mathrm{ND}(m)}$}
\end{align*}

\noindent
When $X(\cdot)$ admits an infinite order Karhunen-Lo\`eve expansion in the sense that 
\begin{equation}
\label{a:Klinfini}
 \sigma X (t)= \sum_{i=1}^{\infty} \zeta_i f_i (t)\mbox{ with } \Var(\zeta_i) = \sigma^2 \mbox{ and}~\forall t\in { M},\ \sum_{i=1}^{\infty}| f_i(t)|^2 = 1\,,
 \tag{${\mathrm{KL}(\infty)}$}
\end{equation}
we say that $X(\cdot)$ satisfies Assumption~\eqref{a:Klinfini}. In this case, the non-degeneracy condition reads: $X$ is a.s. differentiable, for all $p>d+1$ and for all $t_1 \in { M}$ 
\begin{align*}
\notag
&\exists (t_{d+2},\ldots,t_{p})\in({ M}\setminus \{t_1\})^{p -d-1}~\text{ pairwise distinct s.t.}\\
&(X(t_1),\nabla X(t_1), X(t_{d+2}),\dots,X(t_p))\ \mathrm{is\ non\ degenerate}.
\label{a:nonDegenerateInfinite+}
\tag{${\mathrm{\,ND}(\infty)}$}
\end{align*}

A standard result shows that if the covariance function of $X(\cdot)$ is $\mathcal C^0({ M}\times { M})$ on ${ M}$ compact then the Karhunen-Lo\`eve expansion exists (of finite or infinite order). The following key result shows that if $X$ has $\mathcal C^1$ paths almost surely, and satisfies Assumptions \eqref{c1} and \eqref{c3} then the non-degeneracy condition also holds, its proof is in Appendix~\ref{sec:proof_prop:KLdonneND}.

\begin{proposition}
\label{prop:KLdonneND}
Let $X(\cdot)$ be a real valued Gaussian random field having $\mathcal C^1$ paths almost surely, and satisfying Assumptions \eqref{c1} and \eqref{c3}. Then, for all $m>d+1$, Assumption \eqref{a:Klm} implies Assumption \eqref{a:nonDegeneratem+}, and also Assumption \eqref{a:Klinfini} implies Assumption~\eqref{a:nonDegenerateInfinite+}.
\end{proposition}

 \noindent We give some examples: The normalized Brownian motion  $W_t/\sqrt{t}$ satisfies Assumption \eqref{a:Klinfini}; Any Gaussian stationary field with a spectrum that admits an accumulation point satisfies Assumptions \eqref{a:Klinfini}  and \eqref{a:nonDegenerateInfinite+} if differentiable, see for instance \citet[Exercices~3.4 and~3.5]{Azais_Wschebor_09}; Any  Gaussian random field satisfying conditions of \citet[Proposition~3.1]{AW05onthedistribution} has \eqref{a:Klinfini} and~\eqref{a:nonDegenerateInfinite+}. Note that in the applications of Section \ref{sec:examples} all the Gaussian random fields satisfy \eqref{a:Klm} with $m$ finite and explicitly given.

 \subsection{The Karhunen-Lo\`eve estimator of the variance}
 \label{sec:KL_var}
In  practical  applications,  the assumption  that  the variance is known  is too restrictive.   In this section,  we are supposed to observe $\sigma    X(\cdot)$ where $\sigma >0 $ is unknown and~$ X(\cdot) $ satisfies Assumption $\bf (A_{1\text{-}4})$. In particular, $ X(\cdot)$ admits a Karhunen-Lo\`eve expansion of order denoted $m_{\mathrm{KL}}( X)$ possibly infinite. We assume the following fifth condition:
 \begin{equation}
 \label{c4}
\kappa:=m_{\mathrm{KL}}( X)-d-1\geq1\,.
\tag{$\bf A_5$}
\end{equation}
We will make use of the notation $m:=m_{\mathrm{KL}}( X)$ when there is no ambiguity.
\begin{remark}
Using Assumption $\eqref{c3}$ and the fact that the gradient of $X(t)$ is independent of~$X(t)$, it can be shown that $m_{\mathrm{KL}}(X)$ is greater than or equal to $d+1$, by an argument similar to the proof of Proposition~\ref{prop:KLdonneND}. Note Assumption~\eqref{c4} requires at least $d+2$ degrees of freedom, and this extra degree of freedom is the price to estimate the variance.
\end{remark}

\subsubsection*{The infinite order case}
When $ X(\cdot)$ satisfies \eqref{a:Klinfini}, note that, for every integer $p\geq 2 $, from  $(\sigma X(t_{1}), \dots,\sigma X(t_p))$ for conveniently chosen  points $t_{1},\dots,t_p$ (say uniformly at random for instance), one can build an estimator, say~$ \hat \sigma^2_{(p)}$, of the variance $\sigma^2$ with chi-squared distribution  $\sigma^2 \chi^2(p-1)/(p-1)$ under~$\mathds  H_0$. Making~$p$ tend  to infinity, classical concentration inequalities and Borel-Cantelli lemma  prove that~$ \hat \sigma^2_{(p)}$  converges almost surely to $\sigma^2$ under~$\mathds  H_0$. Thus the variance $\sigma^2$ is theoretically  directly observable from the entire path of~$\sigma X(\cdot)$ (though in practical applications one will estimate it by  a~$\chi^2$ with a large number of degrees of freedom). We still denote  by  $ \hat \sigma^2_{t}$ this observation.

 \subsubsection*{The finite order case}
By Proposition~\ref{prop:KLdonneND}, Assumption~\eqref{a:nonDegeneratem+} holds with $m>d+ 1$. Let $t \in \mathrm M$ and let $(t_{d+2}, \ldots,t_m)\in { M}^{ \kappa}$ be as in 
Assumption \eqref{a:nonDegeneratem+}, then the Gaussian vector $( \sigma X^{|t}(t_{d+2}), \ldots, \sigma X^{|t}(t_{m}))$ has for variance-covariance matrix $ \sigma^2 \Sigma$ where~$\Sigma$ is some known  matrix  and an estimator of $\sigma^2$ is 
\begin{equation}
\label{def_sigma_mat}
 \hat   \sigma_t^2 := {\big\| \Sigma ^{-1/2} \big( \sigma X^{|t}(t_{d+2}), \ldots,\sigma X^{|t}(t_{m})\big)\big\|^2  }/\, {  \kappa}\,.
\end{equation}
Direct algebra  shows that 
  $ X^{|t}(\cdot) $ inherits  an order~$ \kappa$ Karhunen-Lo\`eve expansion from the order~$m$ Karhunen-Lo\`eve expansion of $ X(\cdot) $. More precisely, under~\eqref{e:H0},
\begin{equation}
\label{eq:KL_bar}
	\forall  s\neq t \in { M}\,,\quad
 	\sigma  X ^{|t}(s)= \sum_{i=1}^{ \kappa}  \bar \zeta_i \bar f_i (s)\text{ with } \Var( \bar \zeta_i) = \sigma^2\,,
\end{equation}
where equality holds in $L^2$, uniformly in $s$, and it holds
\begin{equation}
\label{eq:sigma_t}
\hat   \sigma_t^2 = \frac1{ \kappa}\sum_{i=1}^{ \kappa} \bar \zeta_i^2\,.
\end{equation}
It shows that $  \hat   \sigma_t^2 $ does not depend on the choice of $t_{d+2}, \ldots,t_m$ in \eqref{a:nonDegeneratem+}.

\nopagebreak
  		    
\begin{proposition} 
   \label{prop:am}  
   Let $ X(\cdot)$ satisfy~Assumptions \eqref{c0}, \eqref{c1}, \eqref{c3} and \eqref{c4}. Let $t \in { M}$, then the following claims are true under the null hypothesis~\eqref{e:H0}.
   \begin{itemize}
    \item[$(i)$] $\hat {\sigma}^2_{t}$ is well defined and follows a  $ \sigma^2{ \chi^2( \kappa)}/{ \kappa}$ distribution;
         \item[$(ii)$] $\hat {\sigma}^2_{t}$ is independent of $( X(t),\nabla  X(t))$;
     \item[$(iii)$] the random field $ X^{|t} (\cdot)/\hat\sigma_{t}$ is independent of the random variable $\hat\sigma_{t}$.
   \end{itemize}
   \end{proposition}

\begin{proof} 
Statement $(i)$ follows from \eqref{eq:sigma_t} and Statement $(ii)$ follows from \eqref{def_sigma_mat} and the independence of~$ X^{|t} (\cdot)$ from $( X(t),\nabla  X(t))$. The last statement is a direct consequence of the independence between the angle and the norm of $\bar\zeta=(\bar\zeta_1,\cdots,\bar\zeta_{ \kappa})$ and \eqref{eq:KL_bar}.
\end{proof}

\subsection{Joint law in the unknown variance case}
We present below the use  of the  estimation $\hat \sigma_{t}$ to modify the spacing test. For fixed $t\in { M}$, by Proposition~\ref{prop:am}, we know that $X(t) \,,  \nabla  X(t) \,, \ {X^{|t}(\cdot)}/{\hat \sigma_{t}}$ and $\hat \sigma_{t}$ are mutually independent. By \eqref{d:R}, we recall that 
\[
R(t):=
\Lambda^{-\frac12}_2(t)\,\Big[\nabla^2X(t) +\Lambda_2(t)\, X(t) + \Lambda_3 (t) \nabla X(t)\Big] \,\Lambda^{-\frac12}_2(t)\,,
\]
As Lemma \ref{lem:helix} shows, $ R(t)$ can be expressed as radial limits of~$X^{|t}(\cdot)$ at point $t$ hence the random variables $\Big\{ X(t) \,,  \nabla  X(t) \,, \ \big({X^{|t}(\cdot)}/{\hat \sigma_{t}}\,,\ {R(t)}/{\hat \sigma_{t}}\big)\,, \ \hat \sigma_{t}\Big\}$ are mutually independent and, by consequence, the variables
\[
  X(t) \,, \nabla  X(t) \,,\ \Bigg[\frac{\lambda_2^t}{\hat \sigma_{t}}, \frac{R(t)}{\hat \sigma_{t}} \Bigg]\,, \ \hat \sigma_{t}
  \text{ are mutually independent,}
\]
where we recall that $\lambda_2^t$ is defined in \eqref{eq:lambdaT}. %
We consider ${t}\in\mathrm M$ a putative value and $\lambda_1^t$ build from $\sigma X(\cdot)$ in \eqref{def:hatz}. Now, consider the test statistics
\begin{equation}
\label{d:tt}
T_{1,t}:=    \frac{\lambda_1^t}{\hat \sigma_{t}}\,,\  T_1:= T_{1, {t_1}}\,,\ 
T_{2,t}:=    \frac{\lambda_2^t}{\hat \sigma_{t}}\,,\  T_2:= T_{2, {t_1}}\,,\ 
\hat\sigma:=\hat\sigma_{{t_1}}\text{ and } \Omega=\sigma R(t_1)\,,%
\end{equation}
as in \eqref{def:R}, and let $ \bar{\mu}_t$  be the joint law of~$(T_{2,t}, {\sigma R(t)}/{\hat \sigma_{t}})$. 

\medskip

Under null~\eqref{e:H0}, note that~the variable $ \sigma X(t)$ is a centered Gaussian variable with variance~$\sigma^2$ and~$\sqrt{ \kappa}\,{\hat \sigma_{t}}/{\sigma}$ is distributed as a $\chi$-distribution with $ \kappa$ degrees of freedom. Hence, ${\hat \sigma_{t}}/{\sigma}$ has density
\[
f_{\frac{\chi_{{ \kappa}}}{\sqrt{{ \kappa}}}}(s) 
= \frac{2^{1-\frac{{ \kappa}}{2}}}{\Gamma \big(\frac{{ \kappa}}{2}\big)} \sqrt{{ \kappa}}  
\,\big(s \sqrt{{ \kappa}}  \big)^{{ \kappa-1}} 
\!\exp\big(-( {{ \kappa}\,s^2}/2)  \big)\,,
\]
under the null~\eqref{e:H0}. Then $( \sigma X(t), {\hat \sigma_{t}}/{\sigma},T_{2,t}, {\sigma R(t)}/{\hat \sigma_{t}})$ has a density \[
 (\mathrm{const}) \, s^{{ \kappa-1}}\exp(-( {{ \kappa}\,s^2}/2))  
 \, \phi(\ell_1/\sigma)\,,
\]
with respect to $\leb_2 \otimes\bar\mu_t$ at point $(\ell_1,s,t_2,r)\in\R^3\times\mathcal S_d$ and where the constant $(\mathrm{const})$ may depend on $m$, $\kappa$ and  $\sigma$.
 Using the same method as for the proof of Theorem~\ref{thm:joint_law} we have the following proposition.

\begin{proposition}
\label{prop:Rice_density_unknown_variance}
Let $ X(\cdot)$ be a Gaussian random field satisfying the set of assumptions~$\bf (A_{1\text{-}5})$. 
Then, under~\eqref{e:H0}, the joint distribution of  $ \big(\lambda_1, { \hat{ \sigma}}/{\sigma},T_2, { \Omega}/{ \hat \sigma} \big)$  has a density with respect to $\leb_2 \otimes \bar \mu^\star$ at point $(\ell_1,s,t_2,r)\in\R^3\times {\mathds{S}_d}$
 equal to
\[
(\mathrm{const}) \det(\ell_1 \I- \sigma s r)\, s^{{ \kappa-1}} \exp(-( {{ \kappa}\,s^2}/2))\, \phi(\ell_1/\sigma) \ind_{\{0<\sigma s t_2 < \ell_1\}},
\]
 where $\bar \mu^\star$ is defined by $\bar  \mu^\star (\cdot) := \int_{ M}  \bar\mu_t(\cdot) \, p_{\nabla X(t)}(0) \det\Lambda_2(t) \,\mathrm d\nu(t)$.
\end{proposition}

\subsection{t-Spacing test, an exact test for the unknown variance case}
\label{sec:tst}
\noindent
We have the second main result, when the variance is unknown.
\nopagebreak
 \begin{theorem}
 \label{thm:rice_unknown_variance} 
 Let $X(\cdot)$ be a Gaussian random field satisfying Assumption $\bf (A_{1\text{-}5})$. For all $r\in\mathcal S_d$, define $  H_r$~as 
\eq
\label{eq:t_spacing}
  \forall \ell\in\bbR,\quad
   H_r(\ell) := \int_{\ell} ^{+\infty}\!\!\!\!\!\det(u  \I- r )\,
   f_{m-1}\Big(u \sqrt{({m-1})/{{ \kappa}}} \Big)\, \mathrm du\,,
\qe
  where $f_{m-1}$ is the density of the Student $t$-distribution with $m-1$ degrees of freedom. Under the null~\eqref{e:H0}, the test statistic  
\[
   \mathbf{T}(T_1,T_2,\Omega/\hat\sigma) :=\frac{ H_{\Omega/\hat\sigma} (T_1)}{H_{\Omega/\hat\sigma} (T_2)}\sim\mathcal U(0,1)\,,
\]
  where $T_1\,$, $T_2\,$, $\Omega$ are given by \eqref{d:tt}. 
\end{theorem}
\begin{proof}
First, using Proposition \ref{prop:Rice_density_unknown_variance} and the change of variable $t_1 = {\ell_1}/{\sigma s}$, we get that the joint density of the quadruplet $(T_1, { \hat{ \sigma}}/{\sigma},T_2, {\Omega}/{ \hat \sigma})$ at point $(t_1,s,t_2,r)$ is given by
\begin{align*}
& (\mathrm{const}) \det(\sigma s (t_1 \I  - r))\, s^{{ \kappa}} 
\exp(-( {{ \kappa}\,s^2}/2))\, \phi(s t_1) \ind_{\{0<t_2 <t_1\}} \\
=&(\mathrm{const}) \det(t_1 \I  - r) \,s^{m-1} 
\exp\Bigg[-\Bigg(s \sqrt{\frac{{ \kappa}}{m-1}}\Bigg)^2 \frac{m-1}{2}\Bigg]\, \phi(s t_1) \ind_{\{0<t_2 <t_1\}}.
\end{align*}
Second, note that if $U$ and $V$ are independent with densities $f_U$ and $f_V$ then ${W}:=\frac{U}{V}$ has density
\[
f_{{W}}(w) = \int_{\R} f_U(w v) v f_V(v) \mathrm{d}v\,.
\]
In our case, integrating over $s$ and with the change of variable $s\leftarrow s\sqrt{{ \kappa}/(m-1)}$, it holds
\begin{align*}
& \int_{\mathds  {R}^+} \phi(s t_1)\, s^{m-1} 
\exp\Bigg[-\Bigg(s \sqrt{\frac{{ \kappa}}{m-1}}\Bigg)^2 \frac{m-1}{2}\Bigg]\, \mathrm{d}s \\
&= (\mathrm{const}) \int_{\mathds  {R}^+} \phi\Bigg(s t_1 \sqrt{\frac{m-1}{{ \kappa}}} \Bigg) \, s\, { s^{m-2} \exp\left[-\frac{s^2 (m-1)}{2}\right]}\, \mathrm{d}s \\
&= (\mathrm{const})\int_{\mathds  {R}^+} \phi\Bigg(s t_1 \sqrt{\frac{m-1}{{ \kappa}}} \Bigg) \, s\,  f_{\frac{\chi_{m-1}}{\sqrt{m-1}}}(s) \, \mathrm{d}s \\
&= (\mathrm{const}) \, f_{m-1}\Bigg(t_1 \sqrt{\frac{m-1}{{ \kappa}}} \Bigg).
\end{align*}
Putting together, the density of $(T_1,T_2,\Omega/\hat{\sigma})$ at point $(t_1,t_2,r)$ is now given by
\[
(\mathrm{const})\det(t_1\I_d  - r)\, f_{m-1}\Big( t_1 \sqrt{\frac{m-1}{{ \kappa}}} \Big) \ind_{\{0<t_2 <t_1\}}\,,
\]
and we conclude using the same trick as the one of Theorem \ref{thm:rice_known_variance}.
\end{proof}

\begin{figure}[!t]
\includegraphics[height=6.3cm]{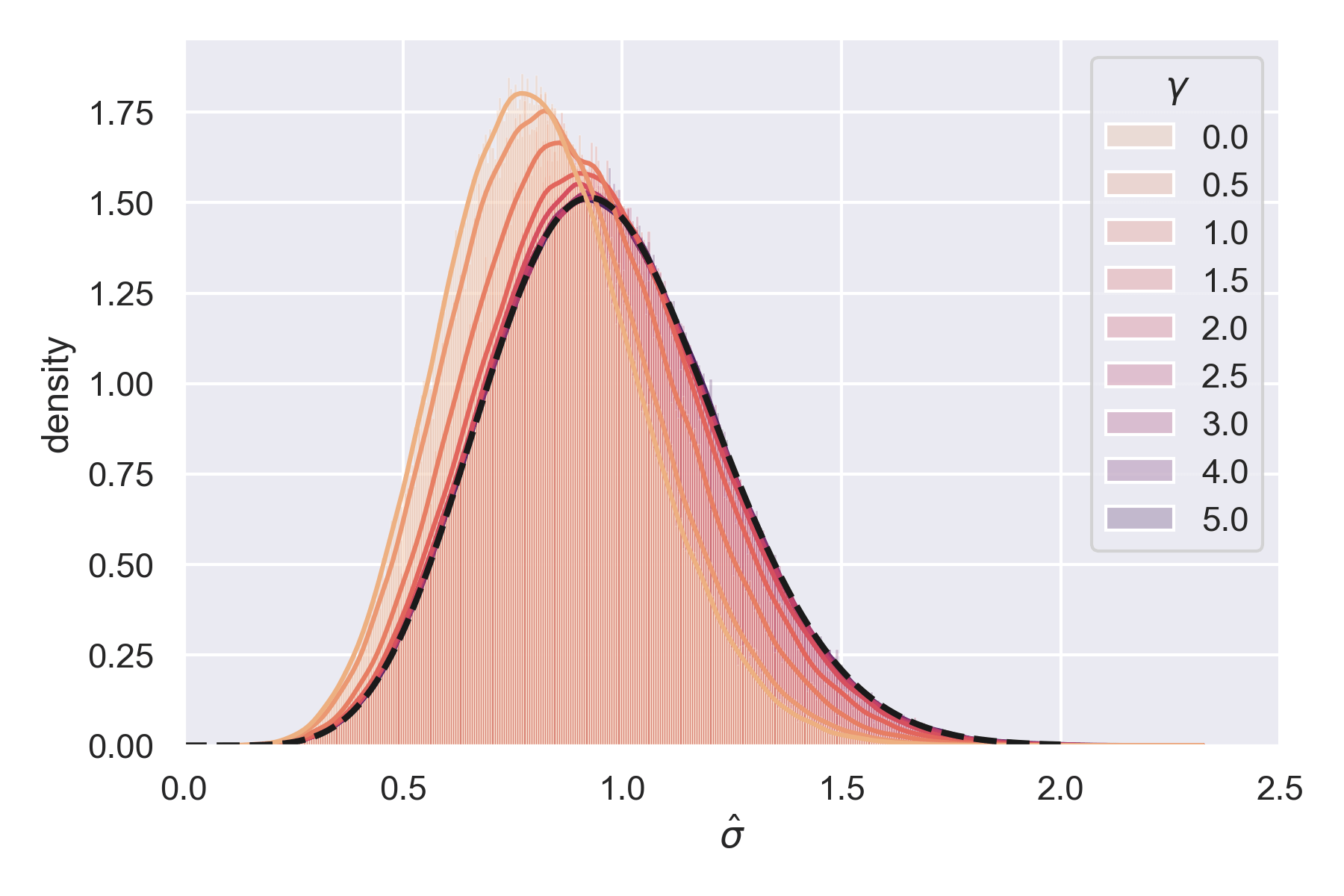}
\caption{
{\bf [Spiked tensor PCA, Section~\ref{sec:spike_tensor_model}, example 4/5]} The variance estimator is not distributed according to $\chi^2$-distribution with~$\kappa=7$ degrees of freedom, it underestimates the variance. The dashed black line is the distribution of a $\chi(7)/\sqrt{7}$, and $\sigma=1$ in these experiments. The probability distribution function is estimated over $250,000$ Monte-Carlo samples for each value of $\gamma$.
}\label{fig:sigma}
\end{figure}

\begin{remark}[On the variance estimator]
The formula \eqref{eq:t_spacing} involves the Student density function with~$m-1$ degrees of freedom which necessitates several comments: 
\begin{itemize}
    \item First, this formula shows that ${\kappa}\,{\hat \sigma^2}/{\sigma^2}$ (resp. $({m-1})\,{\hat \sigma^2}/{\sigma^2}$) fails to be distributed according to a $\chi^2$-distribution with $\kappa$ (resp. $m-1$) degrees of freedom. This because $t_1$ is random. 
    \item Second, Figure~\ref{fig:sigma} illustrates that ${\hat \sigma^2}$ under-estimates the variance $\sigma^2$ under the null ($\gamma=0$) or moderate alternatives ($\gamma\simeq 1$). The reason for this is clear: the chi-squared distribution assumes that the model is pre-specified, not chosen on the basis of $\sigma X(\cdot)$. But the Continuous LARS procedure has deliberately chosen the strongest predictor $\psi_{t_1}$, with $t_1\in M$ maximum of $Z(\cdot)$, among all of the available choices, so it is not surprising that it yields a drop in the value of the variance estimate on the residuals $Z^{|t_1}(\cdot)$.
    \item Third, Figure~\ref{fig:sigma} shows that ${\kappa}\,{\hat \sigma^2}/{\sigma^2}$ has almost $\chi^2$-distribution with $\kappa$ degrees of freedom for large alternatives~($\gamma=5$). The reason for this is clear: the mean $m(\cdot)$ is much larger than noise $\sigma X(\cdot)$ hence $t_1\simeq t_0$  and $\kappa\hat\sigma^2/\sigma^2\simeq\kappa\hat\sigma_{t_1}/\sigma^2$ which has $\chi^2$-distribution with $\kappa$ degrees of freedom. 
\end{itemize}
\end{remark}

\section{Examples}
\label{sec:examples}

\begin{table}[!t]
    \centering
    \label{tab:applications_summary}
    \begin{tabular}{@{}llcc@{}}
        \toprule
        \textbf{Application} & \textbf{Manifold} $\boldsymbol{M}$ & \textbf{Dim.} $\boldsymbol{d}$ & \textbf{Section} \\
        \midrule
        Spiked Tensor PCA (Rank-1) & Sphere $\mathds S^2$ & $2$ & \ref{sec:spike_tensor_model} \\
        Two-Spiked Tensor Model & $\mathds S^1 \times \mathcal{M}_{n,2}$ (Stiefel) & $2n-2$ & \ref{sec:two-spiked} \\
        Super-Resolution & Torus $\mathbb{T}^2 = [0, 2\pi)^2$ & $2$ & \ref{sec:SR} \\
        2-Layer Neural Networks & $\mathds S^{r-1} \times (\mathds S^{n-1})^r$ & $nr-1$ & \ref{sec:neural-net} \\
        \bottomrule
        \\
    \end{tabular}
    \caption{Summary of applications considered. We specify the underlying manifold $M$ and its dimension $d$.
    }
\end{table}

\subsection{How to deploy our method}
\begin{subequations}
We are given an observation $Y\in E$ as in Equation \eqref{eq:def_Y_regression} and we compute the corresponding profile likelihood random field given by $Z(t)=\langle Y,\Psi_t\rangle_E$. We would like to test the global nullity of its mean~$m(\cdot)$ given by hypothesis~\eqref{e:H0}. The testing framework is depicted in Section~\ref{sec:testing_framework} and we recall that
\eq
\label{eq:Omega_examples}
\Omega=\Lambda^{-\frac12}_2(t_1)\,(\nabla^2Z(t_1) +\lambda_1\Lambda_2(t_1))\,\Lambda^{-\frac12}_2(t_1)\,.
\qe
The first and second maximum (and their arguments) are given by \eqref{def:hatz}, which are computed using a Riemannian gradient descent algorithm on $M$. We will give the expression of $\Lambda_2(t)$ in each example together with a closed form expression to compute the Riemannian Hessian from the Euclidean Hessian and the Euclidean gradient. The Euclidean Hessian can be computed using numerical differentiation or, in some cases, is given explicitly.

When the variance $\sigma^2$ is known, the testing statistics is given by Theorem~\ref{thm:rice_known_variance}. In particular, one needs to compute
\eq
\label{eq:G_examples}
    G_{\Omega/\sigma}(\ell)=\int_{\ell}^{+\infty}\!\!\!\!\!\!\det ( u \mathrm{Id}-\Omega/\sigma)    \phi(u )\, \mathrm{d}u\,.
\qe
When the variance $\sigma^2$ is unknown, the testing statistics is given by Theorem~\ref{thm:rice_unknown_variance}. In particular, one needs to compute
\eq
\label{eq:H_examples}
H_{{\Omega/\hat\sigma}}(\ell) = \int_{\ell} ^{+\infty}\!\!\!\!\!\det(u  \I- {\Omega/\hat\sigma} )\,
   f_{m-1}\Big(u \sqrt{({m-1})/{{ \kappa}}} \Big)\, \mathrm du\,,
\qe
where we recall that $f_{m-1}(\cdot)$ is the density of the Student $t$-distribution with $m-1$ degrees of freedom. Numerical integration can be done but, in some cases, we give explicit expressions of~$G_{\Omega/\sigma}(\ell)$ and $H_{{\Omega/\hat\sigma}}(\ell)$. In double precision arithmetic, $G_{\Omega/\sigma}$ must be evaluated on a logarithmic scale when $\ell$ is large; see \eqref{eq:logG} and Figure~\ref{fig:spacing_close_maxima}.

As for the variance estimation, we draw $\kappa$ independent points $(t_{d+2}, \ldots,t_m)$ uniformly on~$M$. They generically satisfy Condition \eqref{a:nonDegeneratem+}. For a putative point $t\in M$, the Gaussian vector $( X^{|t}(t_{d+2}), \ldots, X^{|t}(t_{m}))$ has for variance-covariance matrix $ \sigma^2 \Sigma$ where~$\Sigma$ is some known  matrix  and an estimator of $\sigma^2$ is given by \eqref{def_sigma_mat} with $t=t_1$ as in \eqref{d:tt}. In our experiments, we have drawn $5$ independent samples on $\kappa$ points and we found the same value for $\hat\sigma$, as shown by the theory \eqref{eq:sigma_t}.
\end{subequations}

\subsection{Spiked tensor PCA}
\label{sec:spike_tensor_model}
We consider a simple example of the detection of a rank one $3$-way symmetric tensor observing $Y=\lambda_0 t_0^{\otimes 3}+\sigma W$ where $ W $ is an order $3$ symmetric tensor defined using symmetry and the following independent terms,
\begin{align*}
W_{iii} &\text{ of variance } 1; \text{ there are } 3 \text{ principal diagonal terms}, i\in[3]\,, \\
 W_{iij} &\text{ of variance } 1/3;\text{ there are } 6 \text{ sub-diagonals terms},  i\neq j\in[3]\,,\\
W_{123} &\text{ of variance } 1/6;\text{ there is }  1 \text{ off-diagonal term}\,.
\end{align*}
The profile likelihood is given by $Z(t)=\langle Y,t^{\otimes 3} \rangle_E =\lambda_0 \langle t_0, t\rangle^3 + \sigma X(t)$ where $ M=  S^2$ is the $2$-sphere and $E$ is the Euclidean space of $3$-way symmetric tensors of size $3\times3\times3$ denoted by $(E,\langle\cdot,\cdot\rangle_E)$ with dimension $m=10$. Hence, we have 
\begin{align*}
\psi_t  &:= t^{\otimes 3} =(t_i t_j t_k)_{1\leq i,j,k\leq 3}\,,\\
c(s,t)  &:= \langle s,t\rangle^3\,,\\
Z(t)    &:= \langle Y,\psi_t\rangle_E=\sum_{ijk} t_i t_j t_k Y_{ijk}\,.
\end{align*}
The reader may consult \url{https://github.com/ydecastro/tensor-spacing/} for further details on the numerical experiments. Figure~\ref{fig:eigenvector} illustrates the geometric intuition behind the second maximum in this context. The left panel displays the original random field $X(\cdot)$. The middle and right panels depict the conditional field $X^{|t_1}(\cdot)$, which represents the regression remainder of $X(\cdot)$ with respect to the value and gradient at the maximizer $t_1$. Crucially, the volumetric view (right panel) reveals that $X^{|t_1}(\cdot)$ exhibits a singularity at $t_1$, forming a structure we refer to as the \textit{helix}. This singularity arises because the limit of the conditional field at $t_1$ depends on the direction of approach, a phenomenon theoretically characterized in Lemma~\ref{lem:helix} (see also Remark~\ref{rem:helix}).

\subsubsection*{Hessian}
To compute  gradient and  Hessian  we can take advantage of  the isometry  of the  problem and consider, without loss of generality, the case where $t\in M$ is the so-called ``north pole'': $(0, 0,1)$. We have 
\[
Z'(0,0,1)  =  \left(\begin{array}{c} 3 Y_{133} \\ 3 Y_{233} \\3 Y_{333} 
\end{array}\right)\,.
\]
This shows that  $\Lambda_2=3\,\I_2$ (on the tangent space), hence 
\[
\Omega=\sigma R(t_1)=\frac\sigma3\tilde R(t_1)\,.
\]
Since the second fundamental form  of the unit sphere is $- \I$, the Riemannian Hessian is equal  to the  Euclidean Hessian $Z''(t)$  limited to the tangent space $-\I$  times the normal derivative (oriented outwards). The  Euclidean Hessian $Z''(t)$  limited to the tangent space is given by
 
 \[
 \left(\begin{array}{cc}2 Y_{113}& Y_{123} \\Y_{123} & 2Y_{223}\end{array}\right)\,.
 \]
 This shows that the gradient is independent from the Hessian. To compute the independent part~$R$ we can use  
\[
  Z(0, 0,1) = Z'_3(0, 0,1) =   Y_{333}  \quad \mbox{ and }   \quad Z'_3(0, 0,1)= Y_{333}\,.
\]
This shows that~$R(t)$ is always  equal to the Euclidean Hessian restricted to  the tangent space. By~\eqref{def:R}, it yields
\[
\Omega=\frac{\sigma} 3\tilde R(t_1)=\frac{\sigma} 3(\I-\Pi_{t_1})\,Z''(t_1)\,(\I-\Pi_{t_1})\,,
\]
where $\Pi_t:=tt^\top$ is the orthogonal projection onto the normal space at point $t$, and direct algebra~gives
\eq
\label{eq:Eucl_Hessian_tensor}
\frac{\partial^2Z }{ \partial t^i \partial t^j}(t) =  6 \sum_{k=1}^3 Y_{ijk}t^k\,.
\qe

\subsubsection*{Test statistics} One can compute  $(t_1,\lambda_1)$ and $(t_2,\lambda_2)$ using a gradient descent. The expression of $\Omega$ has an explicit form by~\eqref{eq:Eucl_Hessian_tensor}. Because we deal with $2\times 2$ matrices, we have for the spacing test the following identities, 
\begin{align*}
 G_{\Omega/\sigma}(\ell) 
 	&= 	\int_\ell^{+\infty}\!\!\!\!\! \!\det(u \I_2-\Omega/\sigma) \phi(u)\,\mathrm du \\
 	&=   \int_\ell^{+\infty}\! \!\!\!\!\! \big[ (u^2-1)  - \mathrm{Tr} (\Omega/\sigma)  u  +\det(\Omega/\sigma) +1 \big]\phi(u)\, \mathrm du \\
 	&= 	\ell \phi(\ell) - \mathrm{Tr}(\Omega/\sigma) \phi(\ell) + (\det(\Omega/\sigma) +1 )(1 -\Phi(\ell))\,,\\
\text{with}\qquad
\det(\Omega/\sigma)  
	& =(1/9\sigma^2)\det \big[ (\I_3-\Pi_{t_1}) Z''(t_1) (\I_3-\Pi_{t_1}) +\Pi_{t_1} \big]\\
 \mathrm{Tr}(\Omega/\sigma) 
 	& = (1/3\sigma)\big({\mathrm{Tr}(Z''(t_1))-t_1^\top Z''(t_1)t_1}\big)\,.
\end{align*}

\noindent
As for the $t$-spacing test, one can check that 
\begin{align*}
 H&_{\Omega/\hat \sigma}(\ell) 
 	= 	\int_\ell^{+\infty}\!\!\!\!\! \!\det(u \I_2-\Omega/\hat \sigma)\, f_{m-1}\Big(u \sqrt{({m-1})/{{ \kappa}}} \Big)\, \mathrm du \\
 	&=   \int_\ell^{+\infty}\! \!\!\!\!\! \big[ u^2  - \mathrm{Tr} (\Omega/\hat \sigma)  u  +\det(\Omega/\hat \sigma)  \big]\,f_{m-1}\Big(u \sqrt{({m-1})/{{ \kappa}}} \Big)\, \mathrm du \\
 	&= 	\frac{\kappa\sqrt{m-3}}{(m-2)\sqrt{m-1}} 
		\frac{\Gamma(\frac m2)\Gamma(\frac{m-3}2)}{\Gamma(\frac{m-1}2)\Gamma(\frac{m-2}2)}
		\sqrt{\frac{\kappa}{m-3}}\times\\
    &\quad
		\bigg[
		\ell\sqrt{\frac{m-3}{\kappa}}f_{m-3}(\ell\sqrt{\frac{m-3}\kappa})+1 -F_{m-3}(\ell\sqrt{\frac{m-3}\kappa})
		\bigg]
		\\
	&\quad	- \mathrm{Tr}(\Omega/\hat \sigma) \frac{\kappa\sqrt{m-3}}{(m-2)\sqrt{m-1}}
			\frac{\Gamma(\frac m2)\Gamma(\frac{m-3}2)}{\Gamma(\frac{m-1}2)\Gamma(\frac{m-2}2)}f_{m-3}(\ell\sqrt{\frac{m-3}\kappa}) \\
	&\quad	+ \det(\Omega/\hat \sigma)\sqrt{\frac{\kappa}{m-1}}\Big(1 -F_{m-1}(\ell\sqrt{\frac{m-1}\kappa})\Big)\,,
\end{align*}
with $m=10$, $\kappa=7$, $F_{\alpha}$ (resp. $f_\alpha$) the cumulative distribution function (resp. distribution density) of the Student $t$-distribution with $\alpha$ degrees of freedom and 
\begin{align*}
\det(\Omega/\hat\sigma)  
	& =(1/9\hat\sigma^2)\det \big[ (\I_3-\Pi_{t_1}) Z''(t_1) (\I_3-\Pi_{t_1}) +\Pi_{t_1} \big]\\
 \mathrm{Tr}(\Omega/\hat\sigma)
 	& = (1/3\hat\sigma)\big({\mathrm{Tr}(Z''(t_1))-t_1^\top Z''(t_1)t_1}\big)\,.
\end{align*}

\subsubsection*{The $t$-spacing test} 
In Figures \ref{fig:cdf_spacings} and \ref{fig:lamddas}, the alternative is given by $\lambda_0=\gamma \times\sigma\sqrt{3\log 3+3\log\log 3}$ where $\gamma=1$ corresponds to the so-called phase transition in Spiked tensor PCA as presented in \citet[Theorem 1.3]{perry2020statistical}. In Figure~\ref{fig:lamddas}, the top right panel shows that $\lambda_1$ is stochastically increasing as $\gamma$ increases while the top right panel shows that the distribution of $\lambda_2$ remains unchanged for moderate values of $\gamma$, say $\gamma\geq 2$. It illustrates that the spacing between $\lambda_1$ and $\lambda_2$ grows linearly with~$\gamma$. In the moderate regime, the bottom left panel shows that $\lambda_1$ is distributed as a Gaussian with mean $\lambda_0$ and variance~$\sigma^2$ (dashed black line), which is the distribution of~$Z(t_0)$ under $\mathds H_1(t_0,\lambda_0)$. The bottom right panel shows that $t_1\simeq t_0$, for moderate values of~$\gamma$. It illustrates that $(t_1,\lambda_1)$ is weakly close to the distribution of $(t_0,Z(t_0))$, the ($t$-)spacing tests detect the alternative $\mathds H_1(t_0,\lambda_0)$.

\subsubsection*{Behaviour when the two maxima are close}

Once $(\lambda_2,\Omega)$ is given, the $p$-value of the spacing test is a decreasing function of $\lambda_1$, and it goes to $1$ as the spacing $\lambda_1-\lambda_2$ goes to $0$. So the test does not reject when the two maxima are nearly tied. This costs nothing in terms of level: by Theorem~\ref{thm:rice_known_variance}, the $p$-value is exactly uniform under the null, whatever the spacing. Only the power against alternatives that produce small spacings is lost, and this loss is unavoidable: two nearly tied maxima carry no evidence in favour of a single spike. The left panel of Figure~\ref{fig:spacing_close_maxima} shows the $p$-value against the observed spacing in the Monte-Carlo experiments of Figure~\ref{fig:cdf_spacings}. No rejection at level $5\%$ occurs below a spacing of $0.6$, and $97\%$ of the rejections have a spacing larger than $1.5$. Among the $42{,}577$ draws with a spacing smaller than $0.05$, the smallest $p$-value is $0.80$. The empirical level at $5\%$ under the null is $0.0500$ (perfect matching up to $4$ digits). Under alternatives above the phase transition, near ties disappear: at $\gamma=5$, the smallest spacing over $250{,}000$ draws is~$1.65$.

\medskip

Near ties are also the easy case for the numerical evaluation of the statistic $\mathbf{S}_\sigma$ of Theorem~\ref{thm:rice_known_variance}: the two integrals in the ratio are close and the ratio is nearly $1$. The middle panel of Figure~\ref{fig:spacing_close_maxima} shows the $p$-value as the spacing goes to $0$ at fixed $(\lambda_2,\Omega)$: the direct evaluation of the closed form above and the log-scale evaluation \eqref{eq:logG} agree to about $10^{-14}$. The delicate case is the far tail: the term $1-\Phi(\ell)$ rounds to $0$ for $\ell\gtrsim 8$ and $\phi(\ell)$ underflows for $\ell\gtrsim 39$, so the direct evaluation of the ratio returns $0/0$ (right panel). We therefore evaluate the pivot as $\exp\big(\log G_{\Omega/\sigma}(\lambda_1/\sigma)-\log G_{\Omega/\sigma}(\lambda_2/\sigma)\big)$~with
\eq
\label{eq:logG}
\log G_{\Omega/\sigma}(\ell)
= -\frac{\ell^2}{2}-\log\sqrt{2\pi}
+\log\Big[\ell-\mathrm{Tr}(\Omega/\sigma)
+\big(\det(\Omega/\sigma)+1\big)\sqrt{\tfrac{\pi}{2}}\,\operatorname{erfcx}\big(\ell/\sqrt{2}\big)\Big]\,,
\qe
where $\operatorname{erfcx}(x)=e^{x^2}\operatorname{erfc}(x)$ is the scaled complementary error function, available in standard libraries. This evaluation is accurate for arbitrarily large $\ell$. The Student integral $H_{\Omega/\hat\sigma}$ has polynomial tails and raises no such difficulty.

\medskip
We evaluated the $2.79$ million $p$-values of our experiments with \eqref{eq:logG}: the direct and the log-scale evaluations never differ on the decision at level $5\%$. A fraction $0.013\%$ of the draws gave a $p$-value outside $[0,1]$ or NaN. They all come from the optimisation or the estimation step, never from the pivot itself: two swaps of nearly tied peaks ($\lambda_2>\lambda_1$), $329$ draws where the gradient descent stopped below the true second maximum, and $26$ degenerate draws of the $\kappa$ points used to estimate $\sigma$. All are detected by the check $p\in[0,1]$. Lemma~\ref{lem:helix} explains the second group and gives a further check: the true second maximum satisfies $\lambda_2\geq\lambda_{\max}(\Omega)$, because the radial limits of $\sigma X^{|t_1}$ at $t_1$ are the Rayleigh quotients of $\Omega$. The $329$ draws of the second group fail this inequality by a margin larger than $0.01$ (median $0.76$), while on the other draws it holds up to the tolerance of the gradient descent. The scripts of these experiments are available at \url{https://github.com/ydecastro/tensor-spacing/}.

\begin{figure}[!ht]
\begin{centering}
\includegraphics[height=4.1cm]{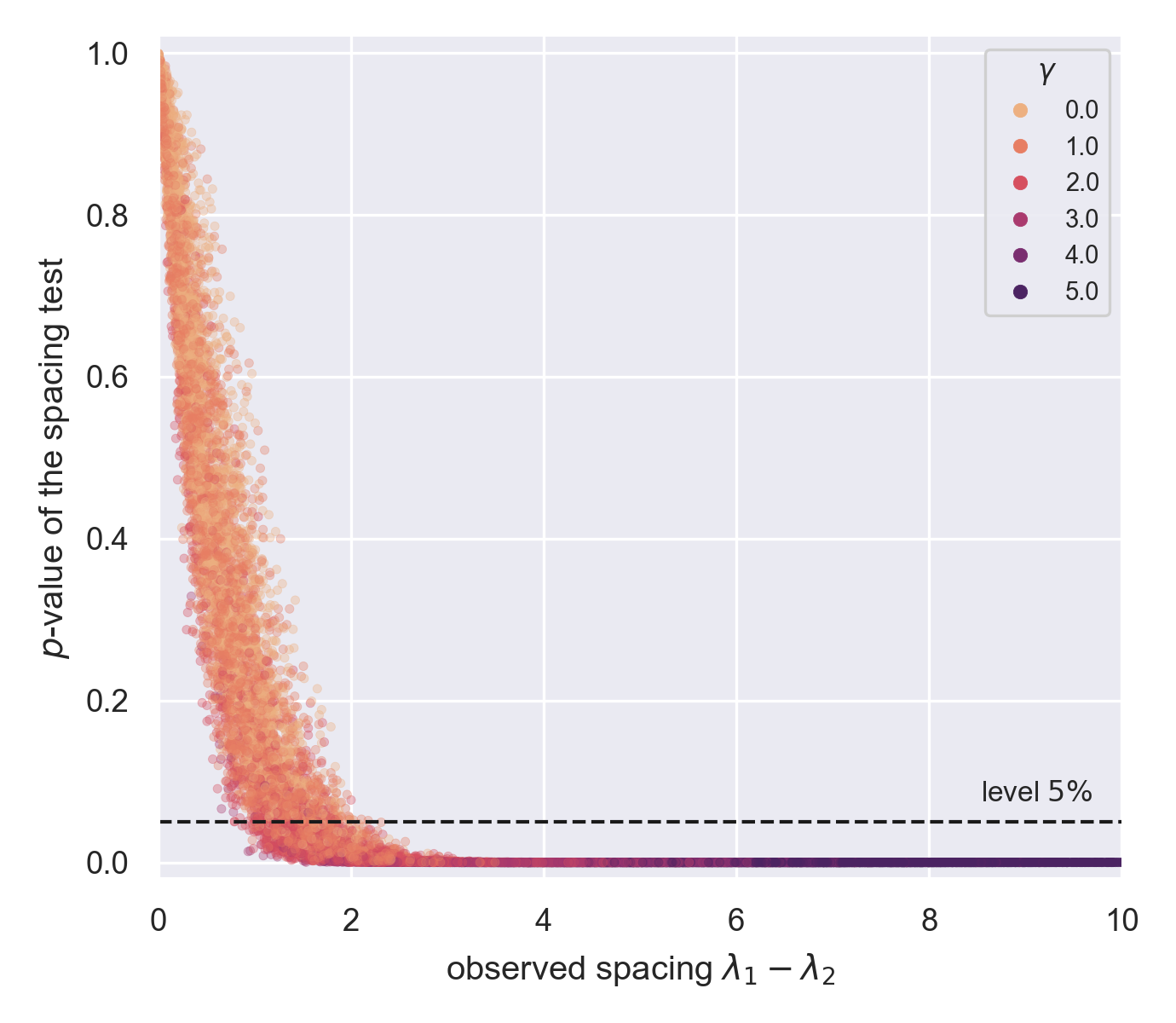}
\includegraphics[height=4.1cm]{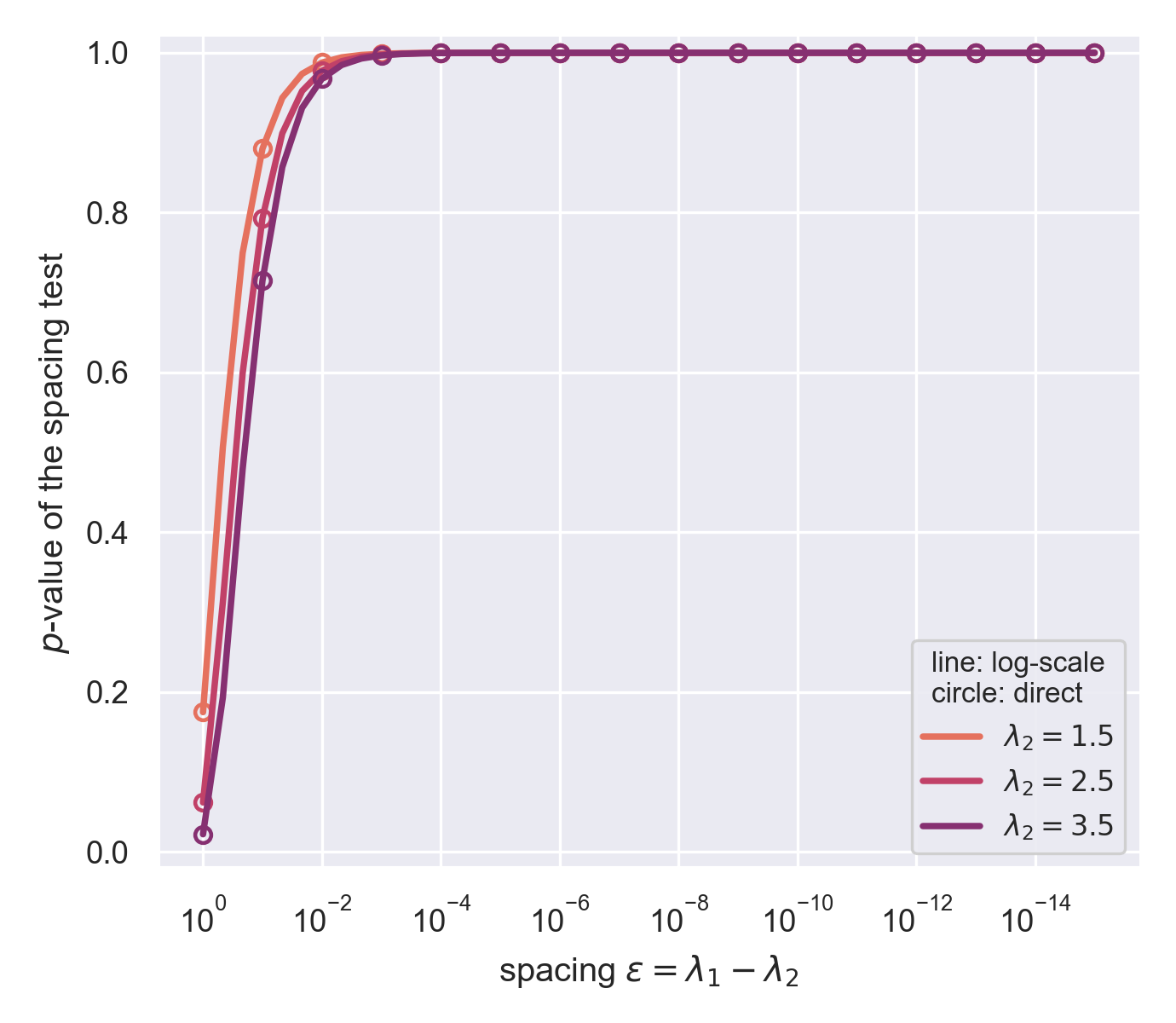}
\includegraphics[height=4.1cm]{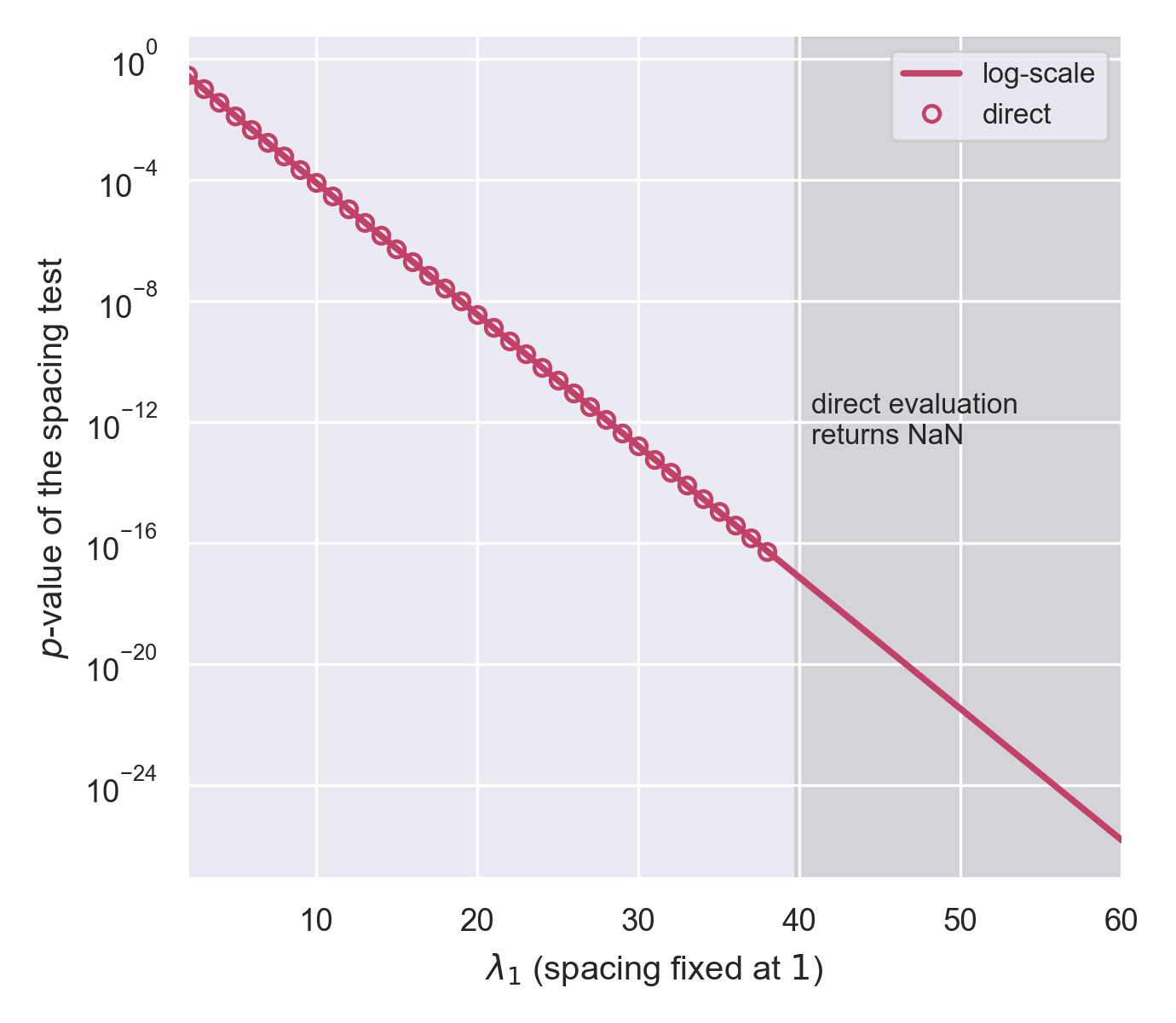}
\par\end{centering}
\caption{{\bf [Spiked tensor PCA, Section~\ref{sec:spike_tensor_model}, example 5/5]} Behaviour of the spacing test when the two maxima are close. Left: $p$-value against the observed spacing $\lambda_1-\lambda_2$ (subsample of $3{,}000$ draws per value of $\gamma$ from the experiments of Figure~\ref{fig:cdf_spacings}). The $p$-value goes to $1$ as the spacing goes to $0$; no rejection at level $5\%$ (dashed line) occurs below a spacing of $0.6$, and $97\%$ of the rejections have a spacing larger than $1.5$. Middle: $p$-value as the spacing $\varepsilon=\lambda_1-\lambda_2$ goes to $0$ at fixed $(\lambda_2,\Omega)$; the direct evaluation (circles) and the log-scale evaluation \eqref{eq:logG} (lines) agree to about $10^{-14}$. Right: $p$-value as $\lambda_1$ grows with the spacing fixed at $1$; the direct double-precision evaluation returns NaN beyond $\lambda_1\approx 39$ (shaded area) while the log-scale evaluation \eqref{eq:logG} follows the exact value.}
\label{fig:spacing_close_maxima}
\end{figure}

\subsection{Two-spiked tensor model}
\label{sec:two-spiked}
We consider a generalization to higher dimensions and two-spiked tensors of the preceding example (Section~\ref{sec:spike_tensor_model}) by 
\[
Y=\nu_{0,1} x_{0}^{\otimes k}+\nu_{0,2} y_{0}^{\otimes k}+\sigma W\,,
\]
where: 
\begin{subequations}
\begin{itemize}
    \item the Euclidean space $E$ is given by $k$-way $n$-dimensional symmetric tensors ($k\geq 3, n\geq 4$) equipped with the dot product
        \eq
            \label{e:dot_product}
            \forall  \ \mathbf T,\mathbf U\in(\R^{n})^{\otimes k}\,,\quad \langle\mathbf T,\mathbf U\rangle_E=\sum_{i_1,\ldots,i_k\in[n]}T_{i_1,\ldots,i_k}U_{i_1,\ldots,i_k}\,,
        \qe
    with Euclidean or Frobenius norm $\|\mathbf T\|^2= \langle\mathbf T,\mathbf T\rangle_E$, where  $[n]= 1,\ldots,n$;
    \item the noise tensor $W\in(\R^{n})^{\otimes k}$ is defined by  
        \eq
            \label{e:noise_tensor}
            W\:= \frac1{k!} %
            \sum_{\pi \in \Perm_n} G^{\pi}\,,
        \qe 
    where $G_{i_1 \cdots i_k}$ for ${1\leq i_1, \ldots, i_k \leq n}$ are i.i.d standard Gaussians, $\Perm_n$ is the set of permutations 
    of the set~$[n]$, and $G^{\pi}_{i_1\cdots i_k} = G_{\pi(i_1) \cdots \pi(i_k)}$. Note that the entries of $W$ with indices $ i_1< i_2<\cdots< i_k$  
    form an i.i.d. collection of Gaussian random variables, namely $\{W_{i_1\cdots i_k}\}_{i_1< i_2<\cdots< i_k}$ with distribution $\mathcal 
    N(0,1/(k!))$.
    \item the eigenvectors $x_0$ and $y_0$ are normalized and orthogonal, they belong to the Stiefel manifold $ {\mathcal M}_{n,2}$ given by  
    \[
        {\mathcal M}_{n,2} := \{ (x,y)\in(\mathds S^{n-1})^2:  x \perp y\}.
    \]
\end{itemize}
\end{subequations}
Using the framework of Section \ref{s:kernel}, we are led to consider the profile likelihood random field 
\begin{align*}
    t       &= (\theta, x,y)\in [0, 2\pi)\times{\mathcal M}_{n,2}\\ 
    \psi_t  &= \cos(\theta)  x^{\otimes k} + \sin(\theta) y^{\otimes k}\\
    Z(t)    &= \big\langle \psi_t, Y\big\rangle_E  = \cos(\theta) \langle x^{\otimes k}, Y \rangle_E+ \sin(\theta) \ \langle y^{\otimes k}, Y \rangle_E\,.
\end{align*}
So the relevant manifold is $M= \mathds S^{1}\times {\mathcal M}_{n,2}$, where the circle $S^1$ is represented by $[0, 2\pi)$. The dimension of $\mathrm{M}$ is $2n-2$. Hence we uncover 
\[
    Y=\lambda_{0} \psi_{t_{0}}+\sigma W\,,
\]
with $\lambda_0=\sqrt{\nu_{0,1}^2+\nu_{0,2}^2}$ and $\theta=\arccos\big({\nu_{0,1}/\lambda_0}\big)$. 

\medskip

\noindent
The covariance function at points $s=(\theta',u,v)$ and $t=(\theta,x,y)$ is given by
\begin{subequations}
\begin{align} 
    \label{e:covF0}
    c(s,t) &= \E\big( Y(t) Y(s)\big) \notag \\ 
    &= \cos(\theta)\cos(\theta') \langle u^{\otimes k},x^{\otimes k} \rangle_E 
    + \sin(\theta) \sin(\theta') \langle v^{\otimes k},y^{\otimes k} \rangle_E\notag \\ 
    &= \cos(\theta)\cos(\theta') \langle u,x \rangle^k 
    + \sin(\theta) \sin(\theta') \langle v,y \rangle^k\,.
\end{align}
In particular, it follows that $c(t,t)=1$ for all $t\in {\mathrm{M}}$ and 
\begin{equation} \label{e:invariancia_c}
c((\theta',u,v),(\theta,x,y)) = c((\theta',Uu,Uv),(\theta,Ux,Uy)), 
\end{equation}
for any orthogonal map $U$ in $\mathds {R}^n$.
\end{subequations}

\subsubsection*{Let us compute the matrix $\Lambda_2(t)$.}

Note that because of the partial invariance by isometry given by \eqref{e:invariancia_c}, $ \Lambda_2(t)$ depends on $\theta$ only. 
Let us define for distinct $ j \neq i,\ell\neq i$:
\begin{align*}
    Y_{[i]} &:= Y_{i,\ldots,i},  \  \mbox{ with variance }  1 ;\\
Y_{[i],j} &:= Y_{i,\ldots,i, j}+ \cdots +  Y_{j, i,\ldots,i}  = k Y_{i,\ldots,i, j},  \mbox{ with variance }  k; \\
Y_{[i],j,\ell} &:= \sum Y_{i_1,i_2,i_3} \mbox{ with } k-2 \mbox{ occurrences of } i, \mbox{ one }j \mbox{ and  one } \ell,\\
&\quad\mbox{with variance }  k(k-1); \\
Y_{[i],j,j}    & \mbox{ has variance }   \frac{k(k-1)}{2}\,.
\end{align*}
Because of the independence properties within $Y$, all these variables are independent. We consider the profile likelihood random field at point $(\theta,e_1,e_2)$ where $e_1$ and $e_2$ are the first two elements of the canonical basis. We have 
\[
 Z(\theta,e_1,e_2) = \cos(\theta) Y_{[1]}  + \sin(\theta) Y_{[2]}. 
\]
The Euclidean gradient at $(\theta,e_1,e_2)$ is given by 
\begin{align*}
  \frac{\partial Z (\theta,e_1,e_2) }{\partial x_1}     &=  k\cos(\theta) Y_{[1]}, 
  &
   \frac{\partial Z (\theta,e_1,e_2) }{\partial y_2}    &=  k \sin(\theta) Y_{[2]},  
  \\
 \frac{\partial Z (\theta,e_1,e_2) }{\partial x_j}      &= \cos(\theta) Y_{[1],j}, j  \neq 1, 
 &
  \frac{\partial Z (\theta,e_1,e_2) }{\partial x_j}     &= \sin(\theta) Y_{[2],j} ,j  \neq 2\,,
\end{align*}
and 
\[
    \frac{\partial Z (\theta,e_1,e_2) }{\partial \theta} = -\sin(\theta) \  Y_{[1]}+ \cos(\theta) Y_{[2]}\,.
\]
Using the following orthonormal parameterization of the tangent space  of the Stiefel manifold:
\[
\left(\begin{array}{cc}
0 &   \frac{-\Delta_{21} } {\sqrt{2}
}\\
 \frac{\Delta_{21} }{\sqrt{2}}& 0 \\
\Delta_{31}&\Delta_{32}\\
\vdots &\vdots\\
\Delta_{n1}&\Delta_{n2}
\end{array}\right),
\]
we get that the Riemannian gradient, including $\theta$, is  given  by
\eq
\label{eq:grad_riem_two_spiked}
    \nabla Z  :=\left(\begin{array}{cc}
    \sin(\theta)  Y_{[1]} + \cos(\theta) Y_{[2]} & \cdot  \\
    \frac{\cos(\theta) Y_{[1],2} - \sin(\theta)  Y_{[2],1} }{\sqrt{2}} & \cdot \\
    \cos(\theta) Y_{[1],3}&\sin(\theta) Y_{[2],3}\\
    \vdots &\vdots \\
    \cos(\theta) Y_{[1],n}&s \  Y_{[2],n}
\end{array}\right)\,.
\qe
Hence, we deduce that
\begin{equation} 
\label{e:Lambda2-2spike}
  \Lambda_2(\theta) = \diag\big(1,k/2, \cos^2(\theta),\ldots,\sin^2(\theta),\ldots \big), 
\end{equation}
where $\cos^2(\theta)$ and $\sin^2(\theta)$ are repeated $n-2$ times.

\medskip

\begin{remark}
Observe that if $\theta = k \pi/2$ then the matrix $\Lambda_2(\theta)$ is singular. We prove in Lemma~\ref{lem:bulinsk} of Appendix~\ref{sec:appendix_two_spiked} that almost surely $\theta_1 \neq k \pi/2$ where $t_1=(\theta_1,x_1,y_1)$, hence $\Lambda_2(t_1)$ is non-singular.

Also, Theorems \ref{thm:rice_known_variance} or \ref{thm:rice_unknown_variance} can not be directly applied. 
But, 
one can retrieve these results by omitting an $\varepsilon$-neighborhood of these points in the {K}ac-{R}ice formula
and passing to the monotonic limit as $\varepsilon$ tends to $0$. %
\end{remark}

\subsubsection*{Let us compute the matrix $\Omega$}
Let us use for short $c$ and $s$ for $\cos(\theta)$ and $sin(\theta)$ respectively. We start by computing the Euclidean Hessian at $(\theta,e_1,e_2)$ as follows
\begin{align*}
\frac{\partial^2 Z (\theta,e_1,e_2) }{\partial x^2_1}               &=    c  k (k-1) Y_{[1]} ,  
&
 \frac{\partial^2 Z (\theta,e_1,e_2) }{\partial y_j \partial y_{j'}} &=    s \  Y_{[2],j,j'}, j,j'  \neq 2, \\
\frac{\partial^2 Z (\theta,e_1,e_2) }{\partial x_1\partial x_{j }}  &=    c  (k-1) Y_{[1],j} , j \neq 1, 
&
  \frac{\partial^2 Z (\theta,e_1,e_2) }{\partial \theta \partial x_1} &=    c   k Y_{[1]}, \\
\frac{\partial^2 Z (\theta,e_1,e_2) }{\partial x_j\partial x_{j' }} &=    c  Y_{[1]j,j'} , j,j' \neq 1, 
&
\frac{\partial^2 Z (\theta,e_1,e_2) }{\partial \theta \partial x_j} &=    c   k Y_{[1],j},j \neq 1, \\
\frac{\partial^2 Z (\theta,e_1,e_2) }{\partial y_2^2 }              &=    s \  k(k-1) Y_{[2]}, 
&
\frac{\partial^2 Z (\theta,e_1,e_2) }{\partial  \theta \partial y_{2}} &=    s \   k Y_{[2]}, \\
\frac{\partial^2 Z (\theta,e_1,e_2) }{\partial y_j \partial y_j'}   &=    s \  (k-1) Y_{[2],j}, j  \neq 2, 
&
\frac{\partial^2 Z(\theta,e_1,e_2) }{\partial \theta \partial y_{j}} &=    s \  Y_{[2],j}, j \neq 2, 
\end{align*}
and 
\[
\frac{\partial^2 Z (\theta,e_1,e_2) }{\partial \theta^2} =   -Z\,.
\]
Note that all the cross derivatives  between $ x$ and $ y$ vanish. The Riemannian Hessian consists of four parts 
(it can be checked by an order two Taylor expansion along the tangent space) namely
\begin{itemize}
\item The projected Euclidean Hessian: the Euclidean Hessian restricted to the tangent space;
\item For each of the three vectors, $V_1,V_2,V_3$, normal to the Stiefel manifold, 
the part of the second fundamental form associated to the vector multiplied by the normal derivative.
\end{itemize}
The 3 components of the second fundamental form are
\begin{align*}
    V_1 & = - \left(
                    \begin{array}{ccc}
                        1/2 & 0         & 0 \\
                        0   &  \I_{n-2} & 0 \\
                        0   & 0         & 0
 \end{array}\right)\\
    V_2 & = - \left(
                    \begin{array}{ccc}
                        1/2 & 0         & 0 \\
                        0   & 0         & 0 \\
                        0   & 0         & \I_{n-2}
 \end{array}\right)\\
    V_3 & = - \left(
                    \begin{array}{ccc}
                        0   & 0                             & 0 \\
                        0   & 0                             & \frac{1}{\sqrt{2}} \I_{n-2} \\
                        0   & \frac{1}{\sqrt{2}}  \I_{n-2}  & 0
                    \end{array}\right)\,.
\end{align*}
Hence, the expression of the Riemannian Hessian $\nabla^2Z(\theta,e_1,e_2)$ is
\[
   \left(\begin{array}{cccc}
 Z(\theta,e_1,e_2)& A& - s  \big(Y_{[1],j}\big)_{j>2}&  c   \big(Y_{[2],j}\big)_{j>2} \\
 \star & B & \frac{1}{\sqrt{2}}   c \big(Y_{[1],2,j} \big)_j &
 - \frac{1}{\sqrt{2}}    s  \big(Y_{[2],1,j} \big)_j\\
  \star & \star &  c \big(  (Y_{[1],j,j'})_{j,j'} - kY_{[1]} \I_{n-2} \big)&
  - \big(  c Y_{[1],2} + s  Y_{[2],1}\big)\I_{n-2}/2
  \\
 \star  & \star & \star  &  s  \big(  (Y_{[2],j,j'})_{j,j'} -k Y_{[2]} \I_{n-2}\big)
 \end{array}\right),
\]
with $A= \frac{1}{\sqrt{2}} \big( cY_{[2],1} - s  Y_{[1],2}\big)$,
and
$B = -\frac{1}{ 2}  \big( c  k Y_{[1]}   +  s  k Y_{[2]} \big) +  \frac{1}{2} \big(  cY_{[1],2,2} + s  Y_{[2],1,1}\big)$.

\medskip 

To obtain $\nabla^2 Z(t)$ for a generic point $(\theta,x,y)$ belonging to $\mathcal M_{n,2}$ 
we have to perform an isometry as in \eqref{e:invariancia_c}. Using \eqref{eq:Omega_examples}, we deduce the expression of $\Omega$.

\subsubsection*{Testing procedures}
The test statistics involve the function $G_{\Omega/\sigma}$ (resp. and $H_{\Omega/{\hat \sigma}}$), recalled in Equation~\eqref{eq:G_examples} (resp. and Equation~\eqref{eq:H_examples}), which is given using an integral. This integration can be done numerically.

\newpage

\bibliography{references_all} 

\newpage 

\appendix

\begin{center}
    {\Large\bf Appendix of Second Maximum of a Gaussian Random Field\\ and Exact ($t$-)Spacing test}
\end{center}
\medskip

\noindent
In this appendix, the Riemannian Hessians are represented by $2$-forms.

\section{Continued examples}

\subsection{Super resolution}
\label{sec:SR}

In Signal processing, the super resolution phenomenon is the ability to distinguish two close sources (Dirac masses) from noisy low frequency measurements given by an optical system. This issue can be tackled using sparse regularization on the space of measures using the so-called Beurling-LASSO, introduced by \citet{de2012exact,azais2015spike,candes2014towards,duval2015exact}. 

\begin{figure}[!ht]
\begin{center}
\includegraphics[height=5.5cm]{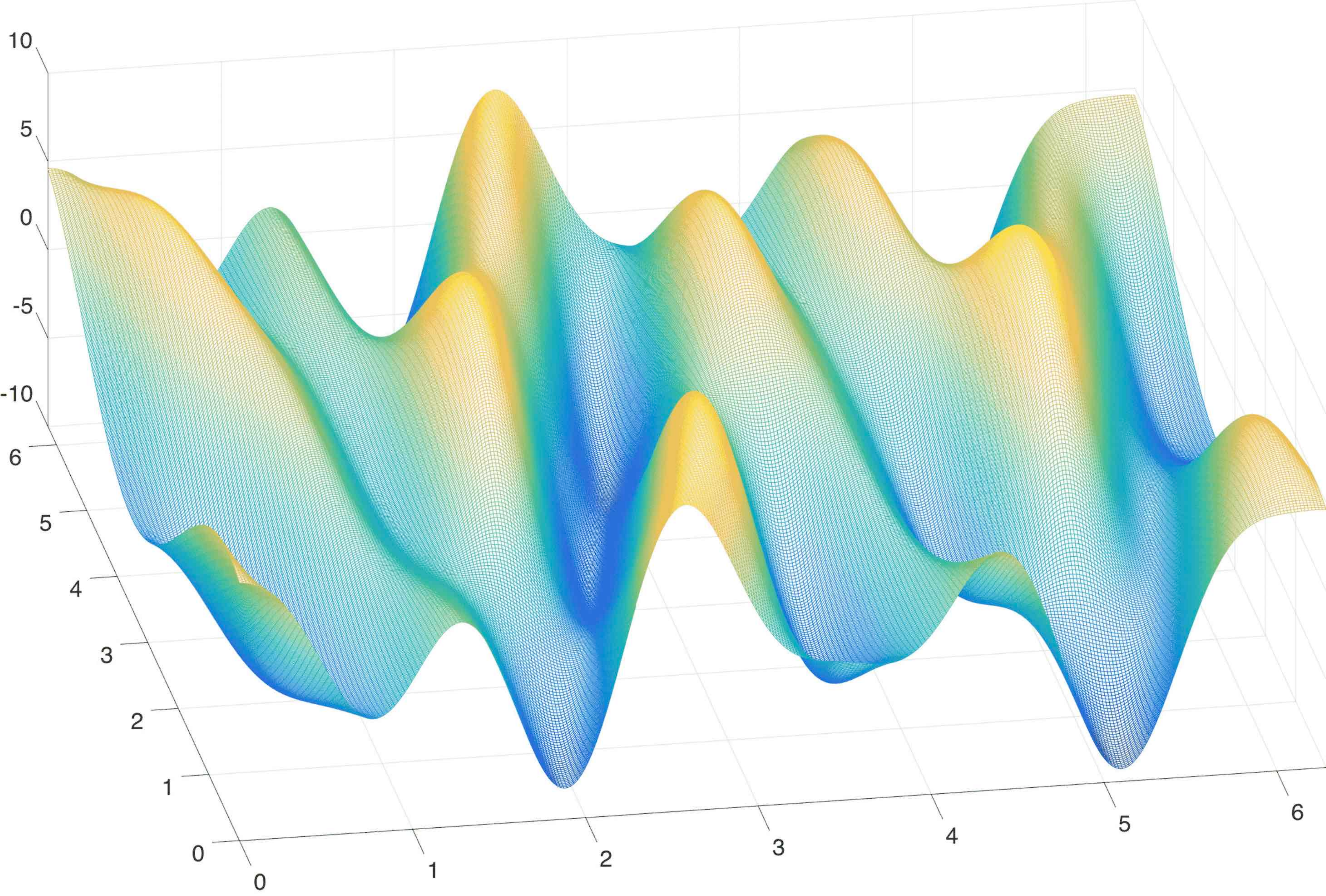}
\end{center}
\caption{{\bf [Super-Resolution, Section~\ref{sec:SR}]} The super-resolution random field $Z(\cdot)$ depicts the observation of a point source at location $x_0$ on the interval $[0,2\pi)$ (spatial domain on the $x$-axis) with phase $\theta_0$ given on the $y$-axis. The $z$-axis corresponds to value of the profile likelihood $Z(x,\theta)$ at point $x$ with phase $\theta$. This figure is presented in \citet[Fig.~2]{azais2020testing}.}
\label{fig:super_resolution}
\end{figure}

We consider the framework of \citet{azais2020testing} where one observes $n=2f+1$ noisy frequencies between $-f$ and $f$ with $f\geq1$. The $\ell^{th}$ Fourier coefficient of a point source (Dirac mass) at location $x_0$ with amplitude $\lambda_0$ and phase $\theta_0$ is $\lambda_0\, e^{\imath \theta_{0}}{e^{-\imath \ell x_{0}}}$. The observation is given by $Y=(y_\ell)_\ell$ where $-f\leq \ell\leq f$, they are complex random variables given by
\[
    y_\ell=\lambda_0\, e^{\imath \theta_{0}}{e^{-\imath \ell x_{0}}}+\sigma W_\ell\,
\]
with $\lambda_0>0$ the amplitude, $t_0=(x_0,\theta_0)\in[0,2\pi)^2$ the location and the phase of the source, and $\sigma>0$ the standard deviation of the noise. The noise is given by $W_\ell=\xi_{\ell,1}+\imath \xi_{\ell,2}$ with~$\xi_{\ell,p}$ independent standard Gaussian variables, $-f\leq\ell\leq f$ and $p=1,2$. 

\medskip

The profile likelihood is given by 
\[
Z(t)=\langle Y,\psi_t \rangle_E
    =\lambda_0 \cos(\theta-\theta_0) \mathbf{D}_f(x-x_0)+\sigma X(t)\,
\]
with $t=(x,\theta)$, $M=[0,2\pi)^2$ is the $2$-Torus, $E=\mathds C^n$ equipped with the standard complex inner product, $X(\cdot)$ is a (real valued) centered Gaussian random field with covariance function $c(\cdot,\cdot)$, $\mathbf{D}_f$ is the Dirichlet kernel. The feature map, the Dirichlet kernel and the covariance function are given by
\begin{align*}
    \psi_t          &= e^{\imath \theta}(e^{\imath f x},\ldots,e^{-\imath f x})\in\mathds C^n\,,\\
    c(t,t')         &= \cos(\theta-\theta') \mathbf{D}_f(x-x')\,,\\
    \mathbf{D}_f(u) &= \frac{\sin(f u/2)}{\sin(u/2)}\,.
\end{align*}
An illustration of the super-resolution (profile likelihood) random field $Z(\cdot)$ is given in Figure~\ref{fig:super_resolution}. One can check that the assumptions of the present article are satisfied by the super-resolution random field. The spacing test is given by 
\[
	\mathbf{S}_\sigma=\frac{
		\sigma(\alpha_1\lambda_1+\alpha_2)\,\phi(\lambda_1/\sigma)+(\alpha_1\sigma^2-\alpha_3^2)\,(1-\Phi)(\lambda_1/\sigma)
		}{
		\sigma(\alpha_1\lambda_2+\alpha_2)\,\phi(\lambda_2/\sigma)+(\alpha_1\sigma^2-\alpha_3^2)\,(1-\Phi)(\lambda_2/\sigma)
		}\,,
\]
where $\alpha_1,\alpha_2,\alpha_3$ have explicit forms given in \cite[Proposition~10]{azais2020testing}, $\lambda_1$ is the maximum of $Z(\cdot)$ and $\lambda_2$ the maximum of $Z^{|t_1}(\cdot)$. The $t$-spacing is explicitly described in \citet[Proposition~11]{azais2020testing}.

\subsection{Two-layer neural networks with smooth rectifier} 	
\label{sec:neural-net}
\index{Neural networks}
We observe a $N$-sample $(x_i,y_i)_{i\in[N]}$, where $x_i\in\mathds R^n$ is the input and $y_i\in\mathds R$ the output. The output is such that $y_i=h(x_i)+\sigma w_i$ for some measurable target function $h\,:\,\mathds R^n\to\mathds R$, standard deviation $\sigma>0$, and $w_i$ standard Gaussian variable.

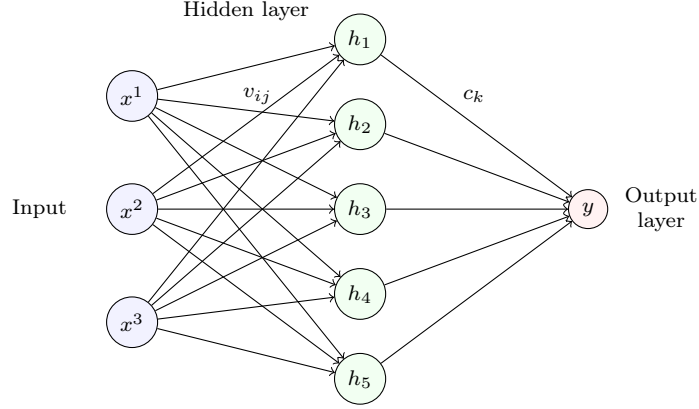
\begin{figure}[!t]
\begin{center}
\begin{tikzpicture}[scale=1.5]
  \foreach \i in {1,...,3}
    \node[circle, draw=black, fill=blue!05] (input-\i) at (0,1-\i) {$x^\i$};
  
  \foreach \h in {1,...,5}
    \node[circle, draw=black, fill=green!05] (hidden-\h) at (2,-0.75*\h  +1.25) {$h_\h$};
  
    \node[circle, draw=black, fill=red!05] (output-1) at (4,-1) {$y$};
  
  \foreach \i in {1,...,3}
    \foreach \h in {1,...,5}
      \draw[->] (input-\i) -- (hidden-\h);
  
  \foreach \h in {1,...,5}
      \draw[->] (hidden-\h) -- (output-1);
  
  \node[align=center, anchor=east] at (-0.5, -1) {Input};
  \node[align=center] at (1, 0.75) {Hidden layer};
  \node[align=center, anchor=west] at (4.25, -1) {Output\\layer};
  \node[align=center] at (1.1,0) {$v_{ij}$};
  \node[align=center] at (3, 0) {$c_k$};
\end{tikzpicture}
\end{center}
\caption{a two-layer neural network with $c_k$ and $v_{i,j}$ real numbers.}\label{f:neurone}
\end{figure}

\noindent
We consider a two layer neural network given by
 \begin{itemize}
 \item  {\bf (hidden layer)} a layer of $r$ neurons with activation a $C^2$ function $\rho(\cdot)$;
    \item  {\bf (output layer)} and a layer  which is simply a mean;
 \end{itemize}
 See Figure \ref{f:neurone} for an illustration. As a consequence the output is $\frac 1r \sum_{k=1} ^r  c_k \rho(\langle v_k,x_i\rangle)$.

\medskip

The unknown parameters are the weights $c_k$ and the directions $v_k$ of the neurons, for $k \in [r]$. We  assume that $v_k \in\mathds S^{n-1}$. Using the trick that the mean of~$r$ reals is the least squares estimator, we have to minimize
\[
 \frac 1N  \sum_{i=1} ^N\Big( y_i -\frac 1r  \sum_{k=1} ^r c_k \rho(\langle v_k,x_i\rangle)\Big)^2  =\frac 1{Nr^2}
  \sum_{i=1} ^N
  \sum_{k=1} ^r \Big( y_i - c_k \rho(\langle v_k,x_i\rangle)\Big)^2\,,
\]
and we recognize the mean-square training error in the left hand side and the Euclidean distance between $Y=(y_i)_{i,k}$ and $(c_k \rho(\langle v_k,x_i\rangle))_{i,k}$ on the right hand side in $E=\mathds R^{N\times r}$ equipped with the dot product 
\[
    \langle (a_{i,k}),(b_{i,k})\rangle_E
        =\frac1N\sum_{i=1}^N
            \Big(\frac1r\sum_{k=1}^ra_{i,k}\Big)
            \Big(\frac1r\sum_{k=1}^rb_{i,k}\Big)\,.
\] 
The Gaussian regression problem is $Y=H+\sigma W$ where $H=(h(x_i))_{i,k}$ and $W=(W_i)_{i,k}$. To enter into the framework of Section \ref{s:kernel} we need to perform the following change of variables
\[
   c_k \rho(\langle v_k,x_i\rangle) =  \lambda  \frac {c_k  V}{\lambda } \frac{ \rho(\langle v_k,x_i\rangle)}{V} =:
    \lambda a_k  \frac{ \rho(\langle v_k,x_i\rangle)}{V},
\]
 with $ V^2 := \sum_{i=1}^N  \rho^2(\langle v_k,x_i\rangle)$ and $ \lambda^2 := V^2(c_1^2 + \cdots  +c_r^2)$ are normalizing constants. One can check that $t= (a_1,\dots,a_r,v_1,\dots,v_r) \in \mathds S^{r-1} \times (\mathds S^{n-1})^{r}$ with $\mathds S^{r-1}$ the Euclidean sphere of $\R^r$ and that ${\mathrm{M}} =\mathds S^{r-1}\times (\mathds S^{n-1})^{r}$. We have 
\[
 \psi_t := \Bigg( \frac{ a_k \rho(\langle v_k,x_i\rangle)}
 { \sqrt{ \sum_{i'}\rho^2(\langle v_k,x_{i'}\rangle)   }}
  \Bigg)_{\substack{k\in[r] \\ i\in[N]}}
  \text{ and }
c(s,t) := \langle \psi_s,\psi_t\rangle_E\,,
\]
and the profile likelihood is $Z(t)=\langle H,\psi_t\rangle_E+\sigma X(t)$ with $X(t)=\langle W,\psi_t\rangle_E$ having covariance function $c(s,t)$.

\subsection{Two-spiked tensor profile likelihood}
\label{sec:appendix_two_spiked}

\begin{lemma}
\label{lem:bulinsk}
With probability $1$, it holds that $\theta_1 \neq k \pi/2  $.
\end{lemma}
\begin{proof}
We prove that $\theta_1  \neq 0 $ a.s., the other cases are equivalent. Note that $\theta_1 =0 $ implies that there exists a point $ x_1$ such that $ f(x_1)$ is the maximum of $f(x) :=\langle x^{\otimes k}, Y\rangle_E$ on $\mathds S^{n-1}$. As a consequence the gradient along $\mathds S^{n-1}$ is zero. Now, note that the derivative with respect to $\theta $ at $\theta =0$  of $Z(\theta, x_1, y)$  is zero  for every $y$ orthogonal  to~$ x_1$. Choosing  an orthonormal basis we obtain $n-1$ such derivatives.
  
Then, we use the Bulinskaya lemma \citep[Prop  6.11]{Azais_Wschebor_09}. We denote by $F(x), x \in \mathds  S^{n-1}$ the random field defined on a set of dimension $n-1$ with values in $\mathds R^{2n-2}$ given by the $2n-2$ derivatives above. To use the Bulinskaya lemma we need to prove that
\begin{itemize}
    \item[{\rm (i)}] the function $F(\cdot)$ has $\mathcal C^1$ paths,
    \item[{\rm (ii)}] the density $p_{F(x)}$ is uniformly bounded.
\end{itemize}
Then since the dimension  of the parameter set is smaller than the dimension of the image set, a.s  there is no  point $x$ such that $F(x) =u$ 
where $u$ is any value and in particular the value $0$.  

\medskip

Condition {\rm (i)} is clear. To address Condition {\rm (ii)}, we can use  the invariance by isometry  and  study the density $p_{F(x)}$  at the particular value $x= e_1$. We can consider the derivatives of  $f(x) $ at $e_1$  along the basis of the tangent space $e_2,\ldots, e_n$. We obtain $Y_{2[1]}, \ldots, Y_{n[1]}$. 
Then we consider  the derivatives in $\theta$  of  $X(\theta, x_1, y)$ with $y = e_2,\ldots, e_n$
We obtain $Y_{[2]}, \ldots, Y_{[n]}$. All these variables are independent with fixed variance so the density of $F(x)$ is bounded. 
\end{proof}

\section{The helix random field}
\label{a:helix}
This section is a proof of Lemma~\ref{lem:helix}. The argument is inspired from \citet[Lemma 4.1]{AW05onthedistribution}. Let $t\in { M}$ be a fixed point. On $L^2$, let $\Pi$ be the projector on the orthogonal complement to the subspace generated by $(X(t), \nabla X(t))$. Consider $h$ a vector of the tangent space at point $t$. Consider the exponential map $\varepsilon\mapsto\exp_t(\varepsilon h)$. This function is well defined on a neighborhood of $0$. Hence there exists $\varepsilon_0>0$ such that for all $\varepsilon\in]\varepsilon_0,\varepsilon_0[$, the point $s(\varepsilon):=\exp_t(\varepsilon h)\in { M}$ exists. The function $\varepsilon\mapsto s(\varepsilon)$ is a parametrization of the geodesic starting at point $t$ with velocity~$h$. We denote by $\varepsilon\mapsto h(\varepsilon)$ the parallel transport of $h$ along this geodesic. By Assumption \eqref{c0}, the Taylor formula of order two gives
\begin{equation}
\label{e:taylor1}
	X(s(\varepsilon))=X(t)+\varepsilon\langle \nabla X(t),h\rangle 
	+ \frac{\varepsilon^2}2\nabla^2 X(s(\varepsilon'))[h(\varepsilon'),h(\varepsilon')]\,,
\end{equation}
for some $\varepsilon'\in[0,\varepsilon]$, and
\begin{equation}
\label{e:taylor2}
	c(s(\varepsilon),t)=1+\frac{\varepsilon^2}2\Lambda_2(t)[h,h]+o(\varepsilon^2)\,,
\end{equation}
by Assumption \eqref{c1}. From \eqref{e:taylor1}, we have 
\[
\Pi(X(s(\varepsilon)))=\frac{\varepsilon^2}2\tilde R(s(\varepsilon'))[h(\varepsilon'),h(\varepsilon')]\,,
\]
Note also that $-\Pi(X(s(\varepsilon)))$ is the numerator of $X^{|t}(s(\varepsilon))$ while the denominator is given by $-\frac{\varepsilon^2}2\Lambda_2(t)[h,h]+o(\varepsilon^2)$ thanks to \eqref{e:taylor2}. We deduce that
\[
X^{|t}(s(\varepsilon))=\frac{\tilde R(s(\varepsilon'))[h(\varepsilon'),h(\varepsilon')]}{\Lambda_2(t)[h,h]+o(1)}
\]
From \eqref{d:Rtecho} and passing to the limit, we deduce that
\[
\lim_{\varepsilon\to0}X^{|t}(s(\varepsilon))=\frac{\tilde R(t)[h,h]}{\Lambda_2(t)[h,h]}\,,
\]
invoking that $\tilde R$ is continuous by regression formulas and Assumption \eqref{c0}, and that $\Lambda_2(t)[h,h]$ is positive by Assumption \eqref{c3}. 

For the second and last statement, observe that 
\[
\limsup_{s\to t}\, X^{|t}(s)
\geq \sup_{\|h\|_2=1}\lim_{\varepsilon\to0}X^{|t}(\exp_t(\varepsilon h))
=\max_{\|h\|_2=1}\frac{\tilde R(t)[h,h]}{\Lambda_2(t)[h,h]}
=\frac{\tilde R(t)[h_0,h_0]}{\Lambda_2(t)[h_0,h_0]}
\]
where the vector $h_0$ exists by continuity of $h\mapsto{\tilde R(t)[h,h]}/{\Lambda_2(t)[h,h]}$ on the Euclidean sphere, which is compact. Now, let $\delta>0$ and let $s_n$ be a sequence such that $s_n\to t$ and 
\[
\lim_{n\to\infty}\, X^{|t}(s_n)\geq \limsup_{s\to t}\, X^{|t}(s)-\delta\,.
\]
Note that the exponential map is a local diffeomorphism on a neighborhood of the point $t$. Hence, there exist a sequence of positive reals $\varepsilon_n$ converging to zero and a sequence $h_n$ of unit norm tangent vectors such that $s_n=\exp_t(\varepsilon_n h_n)$. Since the Euclidean sphere is compact, we can extract a sequence such that $h_n$ converges to unit norm tangent vector $\bar h$. The Taylor formula gives that 
\begin{align*}
X(s_n)&=X(t)+\varepsilon_n\langle \nabla X(t),h_n\rangle 
	+ \frac{\varepsilon_n^2}2\nabla^2 X(\tilde s_n)[\tilde h_n,\tilde h_n]\,,\\
c(s_n,t)&=1+\frac{\varepsilon_n^2}2\Lambda_2(t)[h_n,h_n]+o(\varepsilon_n^2)\,,
\end{align*}
for some $\tilde s_n$ on the geodesic between $t$ and $s_n$ and $\tilde h_n$ the parallel transport at point $\tilde s_n$ of the tangent vector $h_n$ along this geodesic. By the $\Pi$-projection argument above, we deduce that
\[
X^{|t}(s_n)=\frac{\tilde R(\tilde s_n)[\tilde h_n,\tilde h_n]}{\Lambda_2(t)[h_n,h_n]+o_n(1)}\,.
\]
Passing to the limit by continuity yields 
\[
\max_{\|h\|_2=1}\frac{\tilde R(t)[h,h]}{\Lambda_2(t)[h,h]}\geq \frac{\tilde R(t)[\bar h,\bar h]}{\Lambda_2(t)[\bar h,\bar h]}=\lim_{n\to\infty}\, X^{|t}(s_n)\geq \limsup_{s\to t}\, X^{|t}(s)-\delta\,.
\]
Note that the most left hand side term does not depends on $\delta>0$, hence 
\[
\max_{\|h\|_2=1}\frac{\tilde R(t)[h,h]}{\Lambda_2(t)[h,h]}\geq \limsup_{s\to t}\, X^{|t}(s)\,,
\]
which concludes the proof.

\section{Ad-hoc {K}ac-{R}ice formula}
\label{app:weights}
 The paper \citet[Theorem~7.2]{Azais_Wschebor_09}  concerns weighted sum  of number of roots  when  the weight  is a continuous function of time and of the level.   This has been extended, by a monotone convergence argument  to the case of lower semi-continuous weights  in \citet[Section~7]{armentano2023general}). 
 However  this is not sufficient  mainly because  the regularity of $\lambda_2^t$ as a function of $t$ is difficult  to control. For this reason  we must use the following ad-hoc argument.

Denote by $\rm{dist} (s,t)$ the geodesic distance between points $s$ and $t$. Define
\begin{equation} 
	\label{eq:lambdaT_eps}
	\forall t\in { M},\quad
	\lambda_{2,\epsilon}^t:=\sup_{s\in M\text{ s.t. } \mathrm{dist} (s,t) >\epsilon}\sigma X^{|t}(s)
	\quad \text{and} 
	\quad \lambda_{2,\epsilon} := \lambda_{2,\epsilon}^{ {t_1}}\,,
\end{equation}
and note that $\lambda_{2,\epsilon}^t$ is a continuous function of $t$. Define a monotone  approximation $\xi_n(\cdot)$ of the indicator function ${\1}\{\cdot \in B\}$. Then
\begin{equation}
	\label{e:jm:mon}
   	\xi_n( \sigma X(t), \lambda_{2,1/n}^t, \sigma R(t)) \uparrow \1\{(\sigma X(t), \lambda_{2}^t, \sigma R(t)\in B\}\,.
   \end{equation}
   So we can use  \citet[Section~7]{armentano2023general} for instance to compute the expectation of the left side of \eqref{e:jm:mon}. Indeed, all conditions are clear except Condition c) of \citet[Section~7]{armentano2023general} for $\lambda_{2,\epsilon}$ which is detailed  hereunder. And then use monotone convergence theorem gives the result.
  
 \subsubsection*{Checking Condition c) of \citet[Section~7]{armentano2023general}}
  Using regression formulas, the distribution of  $X( \cdot)$  under the condition $X(t_0) =u$  admits the representation 
\[
   X(t)  =  \tilde X(t) +u f(t), 
\]
where $\tilde  X(t) $ corresponds to the distribution conditional on $X(t_0)=0$. Our goal  is to show the continuity of the distribution of $\lambda_2$ in this representation.  The  conditioned random field $   \tilde X(\cdot) $ satisfies 
\[
\forall  \{s\neq t\}\,,\quad  \rm{Cor}(X(s),X(t))<1\,.
\]
 so that, by the Tsirelson theorem,  the maximum  of  $\tilde X(\cdot) +u f(\cdot)$  is a.s.  unique. The first consequence is the continuity, as a function of $u$,   of $\hat  t$. For $t$ fixed, under $X(t_0) =u$,
\[
    X^{|t}(s) = \tilde  X^{|t}(s)  + u  g^{|t}(s),  \quad \rm{dist} (s,t) > \epsilon\,,
\]
with obvious notation. The proof of Lemma  \ref{lem:independent} shows that $ g^{|t}(s)$ is bounded uniformly in $s$ and $ t$. This gives the desired continuity, as a function of $u$,  of the distribution of $\lambda_{2,\epsilon}^t$ and then of $\lambda_{2,\epsilon}$.

\section{Proof of Theorem~\ref{thm:joint_law}}
\label{sec:proof_joint_law}

\begin{proof}
Consider the set of parameters
\begin{equation*}
 {\cal{B}} := \Big\{(\ell_1,\ell_2,r)\in \bbR^2\times \mathcal{S}_{d}\ :\ \ell_2\leq \ell_1\Big\}\,,
\end{equation*}
 where $\mathcal{S}_{d}$ is the set of $d\times d $ symmetric matrices. Let $B_1$ be an open set of $\bbR\times \mathcal{S}_d$ and $b\in\bbR$, define 
 \[
 B=\Big\{(\ell_1,\ell_2,r)\in \bbR^2\times \mathcal{S}_{d}\ :\ \ell_2< \ell_1\,,\ \ell_2<b\,,\ (\ell_1,r)\in B_1\Big\}\,.
 \] 
Note that $B$ is an open set of $\mathcal B$. Note that this class of open sets generates the Borel sigma algebra of $\mathcal B$. Hence, in order to derive the joint law of the triplet $(\lambda_1,\lambda_2, \Omega)$, we compute 
\begin{subequations}
\begin{align} 
 \P\big\{&(\lambda_1,\lambda_2, \Omega) \in B\big\} \notag\\
 &=\P\big\{\exists t\in { M}\ :\ \nabla X(t)=0\,,\ %
 (\sigma X(t),\lambda_2^t, \sigma R(t))\in B\big\}\label{e:krf}\\
 &= \E\big[\#\big\{t\in { M}\ :\ \nabla X(t)=0\,,\ (\sigma X(t),\lambda_2^t, \sigma R(t))\in B\big\}\big]\label{e:krf2}
 \\
  &= \int_{ M} \E\Big[ \big|\det\big(-\nabla^2 X(t)\big)\big| {\1}_{(\sigma X(t),\lambda_2^t, \sigma R(t))\in B} \Big|\nabla X(t) =0\Big] \,
 p_{\nabla X(t)}(0) \,\mathrm d\nu(t)\label{e:krf2b}\\
 &= \int_{ M} \E\Big[ \det\big(-\nabla^2 X(t)\big) {\1}_{(\sigma X(t),\lambda_2^t, \sigma R(t))\in B} \Big|\nabla X(t) =0\Big] \,
 p_{\nabla X(t)}(0) \,\mathrm d\nu(t)\label{e:krf3}\\
 &= \int_{ M} \E\Big[\det(X(t)\Lambda_2 (t)-\tilde R(t)) {\1}_{(\sigma X(t),\lambda_2^t, \sigma R(t))\in B}\Big] \,
 p_{\nabla X(t)}(0) \,\mathrm d\nu(t) \label{e:krf4}\\
 &= \int_{ M} \E\Big[\det(X(t)\mathrm{Id} - R(t)) {\1}_{(\sigma X(t),\lambda_2^t,\sigma R(t))\in B}\Big]\,
 p_{\nabla X(t)}(0) \det\Lambda_2(t) \,\mathrm d\nu(t)\,,\label{e:krff}
\end{align}
\end{subequations}
where $\nu $ is the surfacic measure on ${ M}$ and 
\begin{itemize}
\item we used that $(\sigma X(t),\lambda_2^t, \sigma R(t))\in B$ implies $\lambda_2^t\leq \sigma X(t)$, and invoked Lemma~\ref{lem:selection_event} in \eqref{e:krf};
\item we used the uniqueness of the maximum \eqref{lem:tsirelson} in \eqref{e:krf2};
\item we used a {\it {K}ac-{R}ice formula} to get the third equality \eqref{e:krf2b}, a proof is given in Appendix~\ref{app:weights};
\item we used that $\imath_1=\imath_4$ (see Lemma \ref{lem:selection_event}) and that the Hessian~$-\nabla^2 X(t)$ is semi-definite positive at a maximum to get \eqref{e:krf3};
\item we used the Hessian regression formula \eqref{d:Rtecho} with $\nabla X(t)=0$ and invoked the 
independence of~$\nabla X(t)$ from $(\sigma X(t),\lambda_2^t,\sigma \tilde R(t))$ to get \eqref{e:krf4};
\end{itemize}

\medskip

\noindent
Now, we introduce the measures (defined up to normalizing constants)
\begin{subequations}
\begin{align} 
  \mathrm d\bar\nu (t) &:= p_{\nabla X(t)}(0) \det\Lambda_2(t) \,\mathrm d\nu(t)\,, \label{e:sigmabar}\\
  \mu^\star (\cdot) &:= \int_M \mu_t(\cdot) \,\mathrm d\bar\nu(t)\,, \label{d:measures}
\end{align}
\end{subequations}
where $\mu_t :={\cal{D}} (\lambda_2^t, \sigma R(t))$ is the distribution of  $(\lambda_2^t, \sigma R(t))$. 

\noindent
Since $X(t)$ is independent of~$(\lambda_2^t,  \sigma R(t))$, one has
\[
{\cal{D}} (\sigma X(t),\lambda_2^t, \sigma R(t)) = \mathcal N(0,\sigma^2) \otimes \mu_t\,,
\]
by Assumption \eqref{c1}. Now, coming back to \eqref{e:krff} we have
\begin{align*} 
 \P\{(\lambda_1,&\lambda_2, \Omega) \in B\} 
 = \int_{ M} \E\Big[\det(X(t)\mathrm{Id}-R(t)) {\1}_{(\sigma X(t),\lambda_2^t, \sigma R(t))\in B}\Big] \,
 \,\mathrm d\bar\nu(t)\notag\\
 &= \int_{ M} \int_{\mathds {R}^{D}} \det(\ell_1 \mathrm{Id}/\sigma -\omega/\sigma) {\1}_{(\ell_1,\ell_2, \omega)\in B}
 \varphi(\ell_1/\sigma)\,\mathrm d\ell_1 \,\mathrm d\mu_t(\ell_2,\omega)\,\mathrm d\bar\nu(t)\notag\\
 &= \int_{\mathds {R}^{D}}\int_{ M} \det(\ell_1  \mathrm{Id}/\sigma-\omega/\sigma) {\1}_{(\ell_1,\ell_2, \omega)\in B}
\varphi(\ell_1/\sigma)\,\mathrm d\ell_1\,\mathrm d\mu_t(\ell_2,\omega)\,\mathrm d\bar\nu(t)\notag\\
 &= \int_{\mathds {R}^{D}} \det(\ell_1 \mathrm{Id}/\sigma-\omega/\sigma) {\1}_{(\ell_1,\ell_2, \omega)\in B}
 \varphi(\ell_1/\sigma)\,\mathrm d \ell_1\,\mathrm d\mu^\star(\ell_2,\omega)\,,
\end{align*}
where $D=(d^2+d+4)/2$ is the dimension of $\bbR^2\times \mathcal{S}_{d}$. It implies that the joint density of $(\lambda_1,\lambda_2, \Omega)$ at point $(\ell_1, \ell_2,\omega) $ on the set $\mathcal B$ has  a  density  proportional to $\det(\ell_1 \mathrm{Id}/\sigma-\omega/\sigma)  \varphi(\ell_1/\sigma)$ with respect to ${\leb} \otimes \mu^\star$.  
This implies in turn that the density of the maximum $\lambda_1$ conditional on~$(\lambda_2, \Omega)$ is  proportional to $\det(\ell_1 \mathrm{Id}-\omega)  \varphi(\ell_1/\sigma){\1}_{ \ell_2\leq \ell_1}$. 
\end{proof}

\newpage

\section{Proof of Proposition~\ref{prop:KLdonneND}}
\label{sec:proof_prop:KLdonneND}
\begin{proof}
Since $X$ has $\mathcal C^1$ paths almost surely, note that the covariance function $c(\cdot,\cdot)$ of the Gaussian random field~$X(\cdot)$ has continuous partial derivatives. 

Consider the centered Gaussian vector $V:=(\nabla X(t_1), X(t_1),X(t_{d+2}),\dots,X(t_p))$ with variance-covariance matrix $\Sigma(V)=\E[VV^\top]$ given by
\[
\newcommand\explainA{%
  \overbrace{%
    \hphantom{\begin{matrix}\partial_1\partial_1  c(t_1,t_1) & \partial_1\partial_2  c(t_1,t_1) & \dots & \partial_1\partial_d  c(t_1,t_d)\end{matrix}}%
  }^{\text{$\nabla X(t_1)$}}%
}
\newcommand{\explainB}{%
  \overbrace{%
    \hphantom{\begin{matrix}\partial_1  c(t_1,t_1) & \partial_1  c(t_1,t_{d+2})& \dots&  \partial_1  c(t_1,t_{p})\end{matrix}}%
  }^{\text{$(X(t_1),X(t_{d+2}),\dots,X(t_p))$}}%
}
\newcommand{\explainC}{%
  \left.\vphantom{\begin{matrix}0\\0\\\ddots\\0\\0\end{matrix}}\right\}%
  \text{\scriptsize$m_1$ lignes}%
}
\newcommand{\explainD}{%
  \left.\vphantom{\begin{matrix}0\\0\\0\\0\\0\\0\end{matrix}}\right\}%
  \text{\scriptsize$n-m_1$ lignes}%
}
\settowidth{\dimen0}{%
  $\begin{pmatrix}\vphantom{\begin{matrix}0\\0\\0\\0\\0\\0\\0\end{matrix}}\end{pmatrix}$%
}
\settowidth{\dimen2}{$\explainB$}
\hspace*{-0.5\dimen0}
\begin{matrix}
	\begin{matrix}\hspace*{0.5\dimen0}\explainA&\hspace*{0.5\dimen0}\explainB\hspace*{0.5\dimen0}
	\end{matrix}
	\\[-0.5ex]
	\begin{pmatrix}
		\begin{matrix}
\partial_1\partial_1  c(t_1,t_1) & \partial_1\partial_2  c(t_1,t_1) & \dots & \partial_1\partial_d  c(t_1,t_1)\\
\partial_2\partial_1  c(t_1,t_1) & \partial_2\partial_2  c(t_1,t_1) & \dots & \partial_2\partial_d  c(t_1,t_1) \\
\vdots & \vdots & \ddots & \vdots \\
\partial_d\partial_1  c(t_1,t_1) & \partial_d\partial_2  c(t_1,t_1) & \dots & \partial_d\partial_d  c(t_1,t_1) 
\end{matrix}
&
	\begin{matrix}
\partial_1  c(t_1,t_1) & \partial_1  c(t_1,t_{d+2})& \dots&  \partial_1  c(t_1,t_{p})\\
\partial_2  c(t_1,t_1) & \partial_2  c(t_1,t_{d+2})& \dots&  \partial_2  c(t_1,t_{p}) \\
\vdots & \vdots & \ddots & \vdots \\
\partial_d  c(t_1,t_1) & \partial_d  c(t_1,t_{d+2})& \dots&  \partial_d  c(t_1,t_{p})
\end{matrix}
\\
& \\
\text{\huge$\star$} & 
		\begin{matrix}
 c(t_1,t_1) &   c(t_1,t_{d+2})& \dots&   c(t_1,t_{p})\\
 c(t_{d+2},t_1) &   c(t_{d+2},t_{d+2})& \dots&    c(t_{d+2},t_{p}) \\
\vdots & \vdots & \ddots & \vdots  \\
  c(t_p,t_1) &  c(t_p,t_{d+2})& \dots&   c(t_p,t_{p})
\end{matrix}
\end{pmatrix}

\end{matrix}
\]
where $ \partial_i\partial_j  c(x,x):=\E[(\partial X/\partial x_i )(x)(\partial X/\partial x_j)(x)]$ and $ \partial_i c(x,y):=\E[(\partial  X/\partial x_i )(x) X(y)]$ with $(\partial X/\partial x_i )(x)$ the partial derivative with respect to $x_i$ at point $x=(x_1,\ldots,x_d)\in M$ in Riemannian normal coordinates, and $y\in\mathrm M$. 

Denote by $\mathds H$ the RKHS defined by $ c(\cdot,\cdot)$. By a standard result, see for instance \citet[Corollary 4.36]{steinwart2008support}, it holds that $\partial_i\Phi(t_j)$ belongs to $\mathds H$ and 
\begin{align*}
 c(t_i,t_j)&=\langle \Phi(t_i),\Phi(t_j) \rangle_{\mathds H};\\
\partial_i  c(t_1,t_j)&=\langle \partial_i\Phi(t_1),\Phi(t_j) \rangle_{\mathds H};\\
\partial_i\partial_j  c(t_1,t_1)&=\langle \partial_i\Phi(t_1),\partial_j\Phi(t_1) \rangle_{\mathds H};
\end{align*}
where $\Phi$ is the canonical feature map of $(\mathds H,\langle \cdot,\cdot \rangle_{\mathds H})$. We recognize that $\Sigma(V)$ is the Gram matrix, for the scalar product~$\langle\cdot,\cdot \rangle_{\mathds H}$, of the vector $\Psi$ given by $\Psi:=\big[\partial_1\Phi(t_1)\,\partial_2\Phi(t_1)\,\cdots\,\partial_d\Phi(t_1)\,\Phi(t_1)\,\Phi(t_{d+2})\,\cdots\,\Phi(t_{p})\big]\in\mathds H^{p}$, and it holds $\Sigma(V)=\Psi^\star\Psi$ where $\Psi^\star\,:\, h\in\bbH
\mapsto (\langle\partial_1\Phi(t_1),h\rangle_{\mathds H},\cdots,\langle\Phi(t_{p}),h\rangle_{\mathds H})\in\bbR^p$ is the adjoint operator of $\Psi$.

Using Assumption \eqref{c1}, Assumption \eqref{c3} and \eqref{e:ind_rf_grad}, one gets that $W=( X(t_1),\nabla  X(t_1))$ is non-degenerate, and its variance covariance matrix $\Sigma(W)$ has full rank. This matrix is also the Gram matrix of the system of vectors in~$\mathds H$ given by $\tilde\Psi :=\big[\Phi(t_1) \partial_1\Phi(t_1) \cdots \partial_d\Phi(t_1)\big]\in\mathds H^{d+1}$. We deduce that is a free system ({\it i.e.,} it spans a vector space of dimension $d+1$). One also knows that the dimension of $\mathds H$ is exactly the order of the Karhunen-Lo\`eve expansion, actually it is standard to prove the Karhunen-Lo\`eve expansion from Mercer's theorem. Hence, it is always possible to complete~$ \tilde\Psi$ by $p-d-1$ vectors of the form $\Phi(t)$ to get a $p$ dimensional free system~$\Psi$, otherwise~$\mathds H$ would be of dimension less than~$p$.
\end{proof}

\newpage

\section{Notation}
\label{app:notation}

\begin{table}[!h]
	\newcommand{\ttile}[1]{\scshape#1}
	\centering
	{\small
	\begin{tabular}{ll}
		\multicolumn{2}{c}{\textit{General notation}} \\
		\midrule
        $[a]$                                   & Set of integers $\{1,\ldots,a\}$;\\
		$A^\top$ $(\text{resp. } A_{i:})$       & transpose of a matrix $A$ (resp. $i$-th line);\\
        $\mathrm{Id}$                           & Identity matrix of dimension $d\times d$;\\
        $\mathcal S_d$                          & Space of $d\times d$ symmetric matrices;\\
        $\mathds S^{d-1}$                       & Euclidean sphere of $\mathds R^d$;\\
        $t$ (and $s$)                           & Generic value for a vector on the manifold $M$;\\
        $\mathcal C^k$                          & Set of $k$ times differentiable functions; \\
        $\delta_{jk}$                           & Kronecker symbol;\\
	$\mathrm{(const)}$                      & Positive constant which may change from line to line;\\
        $\mathcal D(U\,|\, V)$                  & Conditional distribution of $U$ with respect to $V$;\\
        $\Var(U)$                               & Variance matrix of a random vector $U$; \\
        $\Cov(U,V)$                             & Covariance matrix of the random vectors $U$ and $V$;\\
        $\phi(\cdot)$                           & Standard Gaussian density in $\R$; \\
        $p_U(\ell) $                            & Density function of the random variable/vector $U$ at point $\ell$;\\
        $\mathds 1_{A}$ (resp. $\mathds 1{\{A\}})$  & Indicator function of condition $A$ (resp. event $A$);\\
        $\mathrm{Leb}$                          & Lebesgue measure on $\mathds R$;\\
        $\mathcal U(0,1)$                       & Uniform law on $(0,1)$;\\
        $\mu\otimes \nu$                        & Product of measures;\\
    &   \\
		\multicolumn{2}{c}{Random fields} \\
		\midrule
        $Z(\cdot)$                              & Observed Gaussian random field, indexed by $M$;\\
		$m(\cdot)$                              & Mean function of $Z(\cdot);$\\
        $\sigma$                                & Standard error of the $Z(\cdot)$;\\
        $X(\cdot)$                              & Centered Gaussian random field with unit variance;\\
          $ c(s,t) $                            & Covariance function of $X(\cdot)$, $c(s,t)=\E(X(s)X(t)) $;\\
    &   \\
		\multicolumn{2}{c}{Differential geometry} \\
		\midrule
        $M$                                     &$\mathcal C^2$-compact Riemannian manifold of dimension $d$ without boundary; \\
          $\nabla f(t) $                        & Riemannian gradient at point $t\in M$ of $f\,:\,M\to\mathds R$; \\
        $\nabla_t g(s,t) $                      & Riemannian gradient at point $t\in M$ of $g(s,\cdot)$; \\
        $\nabla^2 f(t) $                        & Riemannian Hessian at point $t\in M$ of $f\,:\,M\to\mathds R$; \\
    &   \\
		\multicolumn{2}{c}{Continuous regression} \\
		\midrule
        $(E,\langle\cdot,\cdot\rangle_E)$       & Euclidean space;\\
        $\psi$                                  & Feature map from $M$ to $E$, $\psi\,:\,t\in M\mapsto \psi_t\in E$;\\
        $\psi_t$                                & Vector $\psi(t)\in E$;\\
        $J\psi_t$ (resp. ${J\psi_t}^\top$)      & (resp. transpose of) Jacobian matrix of $\psi$ at point $t$;\\
        $W$                                     & Standard Gaussian random vector of $E$;\\
        $\mathbf{S}_\sigma$ (resp. $\mathbf{T}$) 
                                                & Generalized spacing test (resp. Generalized $t$-spacing test);\\
    &   \\
  \end{tabular}
  }
\caption{List of notation: Riemannian gradients and Hessians are defined using the Levi-Civita connection. Riemannian Hessians are represented in matrix form.\label{tab:notation}}
\end{table}
\end{document}